\documentclass[12pt]{article}
\usepackage{amssymb, amsfonts, amsthm, amsmath}
\usepackage{dsfont}
\usepackage[utf8]{inputenc}
\usepackage[english]{babel}
\usepackage{hhline}
\usepackage{lmodern}
\usepackage{array}
\usepackage{stmaryrd}
\usepackage{subfig}
\usepackage{amsfonts}
\usepackage{enumerate}
\usepackage{enumitem}
\usepackage{bbm}
\usepackage{svg}
\usepackage{import}
\usepackage{mathtools}
\mathtoolsset{showonlyrefs}
\usepackage{hyperref}
\usepackage{pgf,tikz,pgfplots}
\usepackage{mathrsfs}

\pgfplotsset{compat=1.15}
\usetikzlibrary{arrows}
\usetikzlibrary{decorations, calc}
\usetikzlibrary{decorations.pathreplacing}
\tikzset{individu/.style={draw,thick}}

\theoremstyle{plain}
\newtheorem{theorem}{Theorem}
\newtheorem{corollary}[theorem]{Corollary}

\newtheorem{lemma}[theorem]{Lemma}
\newtheorem{proposition}[theorem]{Proposition}

\theoremstyle{definition}
\newtheorem{definition}[theorem]{Definition}

\theoremstyle{remark}
\newtheorem{remark}[theorem]{Remark}

\newcommand{\Z}{\mathbb{Z}}
\newcommand{\R}{\mathbb{R}}

\newcommand{\D}{\mathbb{D}}

\newcommand{\calL}{\mathcal{L}}

\newcommand{\calB}{\mathcal{B}}
\newcommand{\calW}{\mathcal{W}}

\newcommand{\calM}{\mathcal{M}}

\newcommand{\cW}{\mathcal{W}}

\newcommand{\indset}[1]{\mathbf{1}_{\left\{#1\right\}}}
\newcommand{\ind}[1]{\mathbf{1}_{#1}}

\DeclareMathOperator{\E}{\mathbb{E}}
\renewcommand{\P}{\mathbb{P}}

\DeclareMathOperator{\Walks}{Walks}
\DeclareMathOperator{\Bin}{Bin}

\renewcommand{\d}{\mathrm{d}}
\newcommand{\calP}{\mathcal{P}}

\renewcommand{\tilde}[1]{\widetilde{#1}}
\renewcommand{\epsilon}{\varepsilon}

\newcommand{\Addresses}{{
  \bigskip
  \footnotesize

  \textsc{Université Paris-Saclay, Laboratoire de mathématiques d’Orsay, 91405, Orsay, France.}\par\nopagebreak
  \textit{E-mail address}: \texttt{paul.melotti at universite-paris-saclay.fr}

 \medskip

  \textsc{Université Paris-Saclay, CNRS, CEA, Institut de physique théorique, 91191 Gif-sur-Yvette, France}\par\nopagebreak
  \textit{E-mail address}: \texttt{sanjay.ramassamy at ipht.fr}

 \medskip

  \textsc{Department of Mathematics, Uppsala University, PO Box 480,
SE-751~06 Uppsala, Sweden}\par\nopagebreak
  \textit{E-mail address}: \texttt{paul.thevenin at math.uu.se}

}}

\title{Points and lines configurations for perpendicular bisectors of convex cyclic polygons}
\author{Paul Melotti \and Sanjay Ramassamy \and Paul Th\'evenin}
\date{\today}

\begin{document}

\maketitle

\begin{abstract}
We characterize the topological configurations of points and lines
that may arise when placing $n$ points on a circle and drawing the $n$
perpendicular bisectors of the sides of the corresponding convex cyclic $n$-gon. We
also provide exact and asymptotic formulas describing a random realizable configuration, obtained either by sampling the points uniformly at random on the circle or by sampling a realizable configuration uniformly at random.
\end{abstract}

\section{Introduction}
\label{sec:intro}

Let $n\geq3$ and let $P_1,\ldots,P_n$ be $n$ distinct points on the
unit circle, arranged in positive cyclic order. For all $i$ between $1$ and $n$, denote by $L_i$ the perpendicular bisector of the
segment $[P_i,P_{i+1}]$, with indices taken modulo $n$. These $n$ lines all go through the center of the circle. We assume that
the points are in \emph{generic position}, which implies in particular that these lines are all distinct and that no point lies on a line. Hence, the $n$ lines divide the plane into $2n$ regions. What are the circular arrangements of points in the $2n$ regions that can be
realized?

Our interest for this question comes from our work~\cite{MRT} on the flip property for $s$-embeddings, a geometric integrable system. We were led to consider configurations of three hyperbola branches $B_1$, $B_2$, $B_3$ and three points $P_1$, $P_2$, $P_3$ such that the foci of $B_i$ are $P_{i-1}$ and $P_{i+1}$, with $1\leq i \leq 3$ and indices taken modulo $3$. In the limit when the hyperbola branches degenerate to the perpendicular bisectors of their foci, we recover the problem described in the previous paragraph with $n=3$.

The rest of this introduction consists in the presentation of our results and is organized as follows. In Subsection~\ref{subsec:deterministic} we answer the question from the first paragraph by characterizing the topological configurations of points and lines arising from the $n$ perpendicular bisectors of a convex cyclic $n$-gon. We also enumerate such configurations. A natural question to ask is what a typical realizable configuration looks like. We tackle this question under two angles: configurations coming from $n$ uniform random points on the circle (Subsection~\ref{subsec:uniformpoints}) or configurations chosen uniformly at random among all realizable configurations (Subsection~\ref{subsec:uniformbracelets}). In Subsection~\ref{subsec:discussion} we compare these two approaches and we state some open questions.

\subsection{Deterministic results}
\label{subsec:deterministic}

We first characterize the realizable configurations. We number in counterclockwise order the $2n$ regions defined by the perpendicular bisectors, the first one being the region that contains $1$. For every $1\leq i \leq 2n$, we set $v_i$ to be the number of points inside the $i$th region. It is not hard to see that each region contains at most one point, since two consecutive points are separated by the perpendicular bisector of the segment connecting them. The word $\underline{v}=(v_1,\ldots,v_{2n})\in\{0,1\}^{2n}$ is called the
\emph{occupancy word} of the collection of points $P_1,\ldots,P_n$. In order to characterize the occupancy words that
may arise as one lets the positions of the points vary, we introduce the notion of
\emph{signature} of a word in $\{0,1\}^{2n}$. If $\underline{v}=(v_1,\ldots,v_{2n})\in\{0,1\}^{2n}$ is an arbitrary word, its signature $\underline{\sigma}=(\sigma_1,\ldots,\sigma_{n})\in\{0,1,2\}^{n}$ is defined by $\sigma_i=v_i+v_{i+n}$ for every $1\leq i \leq n$. See an example in Figure~\ref{fig:ex}.

\begin{figure}[h]
  \centering
  
\definecolor{qqzzqq}{rgb}{0.,0.6,0.}
\definecolor{zzttqq}{rgb}{0.6,0.2,0.}
\definecolor{qqccqq}{rgb}{0.,0.8,0.}
\definecolor{ffffff}{rgb}{1.,1.,1.}
\begin{tikzpicture}[line cap=round,line join=round,>=triangle 45,x=1.0cm,y=1.0cm]
\clip(3,-5) rectangle (11,3);
\draw [line width=1.pt] (6.98,-1.22) circle (3.540903839417274cm);
\draw [line width=1.pt,domain=3:11.48] plot(\x,{(--15.038779343315106-2.1609954882560247*\x)/0.036859971075367426});
\draw [line width=1.pt,domain=-4.3:22.48] plot(\x,{(--11.814498946732883-1.987984568659515*\x)/1.6898633955004378});
\draw [line width=1.pt,domain=-4.3:22.48] plot(\x,{(-0.010610122873524475-0.3781423037709781*\x)/2.172166723930288});
\draw [line width=1.pt,domain=-4.3:22.48] plot(\x,{(-5.6971218567365085--0.5610480277447114*\x)/1.4598414943265776});
\draw [line width=1.pt,domain=-4.3:22.48] plot(\x,{(-3.71889073704037--0.43928660684697984*\x)/0.5349755911872531});
\draw [line width=1.pt,domain=-4.3:22.48] plot(\x,{(-7.468600476095798--0.9513811944655446*\x)/0.6786555235461442});
\draw [line width=1.pt,domain=-4.3:22.48] plot(\x,{(-10.58523073838803--1.4540954069151653*\x)/0.3570859000985047});
\draw [line width=1.pt,domain=-4.3:22.48] plot(\x,{(-18.438212479438555--2.876520021523472*\x)/-1.3441780908158032});
\draw [line width=1.pt,domain=-4.3:22.48] plot(\x,{(--19.0653881205248-1.7552088968093553*\x)/-5.58527050884877});

\draw [fill=black] (8.00299060971339,2.1699100596384886) circle (2.5pt);
\draw[color=black] (8.19,2.48) node {$P_1$};
\draw [fill=black] (5.841995121457365,2.133050088563121) circle (2.5pt);
\draw[color=black] (6.03,2.46) node {$P_2$};
\draw [fill=black] (3.85401055279785,0.4431866930626833) circle (2.5pt);
\draw[color=black] (3.75,0.76) node {$P_3$};
\draw [fill=black] (3.4758682490268717,-1.7289800308676047) circle (2.5pt);
\draw[color=black] (3.67,-1.4) node {$P_4$};
\draw [fill=black] (4.036916276771583,-3.1888215251941823) circle (2.5pt);
\draw[color=black] (4.23,-2.86) node {$P_5$};
\draw [fill=black] (4.476202883618563,-3.7237971163814354) circle (2.5pt);
\draw[color=black] (4.67,-3.4) node {$P_6$};
\draw [fill=black] (5.427584078084108,-4.40245263992758) circle (2.5pt);
\draw[color=black] (5.61,-4.08) node {$P_7$};
\draw [fill=black] (6.881679484999273,-4.759538540026084) circle (2.5pt);
\draw[color=black] (6.77,-4.44) node {$P_8$};
\draw [fill=black] (9.758199506522745,-3.415360449210281) circle (2.5pt);
\draw[color=black] (9.59,-3.06) node {$P_9$};
\draw [fill=ffffff] (5.957009390286611,-4.609910059638489) circle (2.5pt);
\draw [fill=ffffff] (8.118004878542635,-4.573050088563122) circle (2.5pt);
\draw [fill=ffffff] (10.105989447202152,-2.883186693062683) circle (2.5pt);
\draw [fill=ffffff] (10.484131750973129,-0.7110199691323953) circle (2.5pt);
\draw [line width = 1pt] (10.504131750973129,-1.1510199691323953) -- (10.604131750973129,-1.1510199691323953) ;
\draw (10.604131750973129,-1.1510199691323953) node [right] {$1$};
\draw [fill=ffffff] (9.923083723228418,0.7488215251941823) circle (2.5pt);
\draw [fill=ffffff] (9.483797116381439,1.2837971163814355) circle (2.5pt);
\draw [fill=ffffff] (8.532415921915893,1.9624526399275797) circle (2.5pt);
\draw [fill=ffffff] (7.078320515000728,2.3195385400260844) circle (2.5pt);
\draw [fill=ffffff] (4.201800493477256,0.9753604492102812) circle (2.5pt);
\end{tikzpicture}

  \caption{An example with $n=9$. The black dots correspond to the points $P_1,\dots,P_9$ and the white dots correspond to the antipodes of the black dots.
    Here $\underline{v} = (0,0,0,0,1,0,1,0,1,1,0,1,1,1,1,0,0,1)$ is the occupancy word (which counts the number of black dots in each region) and its signature is $\underline{\sigma} = (1,0,1,1,2,1,1,0,2)$. The signature counts the number of black and white dots in each region.}
  \label{fig:ex}
\end{figure}
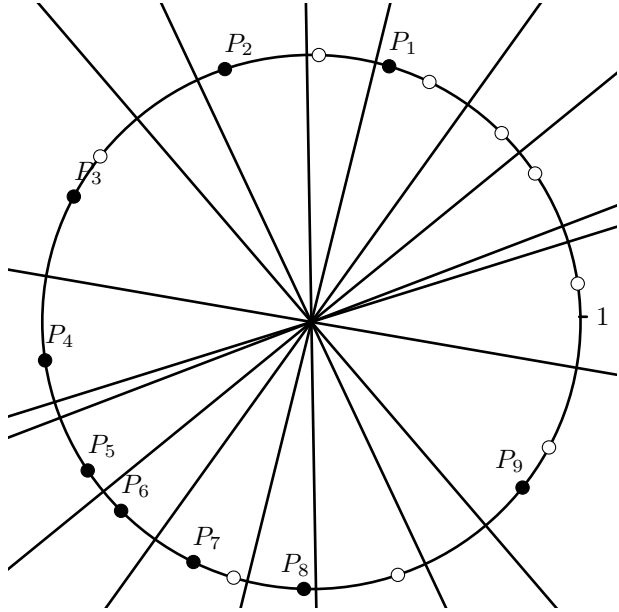

We now introduce a notion of discrete circular interval. Let $N\geq1$ be an integer and let $1\leq i,j\leq N$ be two integers. We define
\[
I_N(i,j)=\begin{cases}
\{i+1,i+2,\ldots,j-1\} &\text{ if } i< j \\
\{1,2,\ldots,j-1\}\cup \{i+1,i+2,\ldots,N\} &\text{ if } i\geq j. \\
\end{cases}
\]
In particular when $i=j$ we have $I_N(i,i)=\{1,\ldots,N\}\setminus\{i\}$.

A word $\underline{s}=(s_1,\ldots,s_{n})\in\{0,1,2\}^{n}$ is called \emph{interlacing} if it satisfies the following two properties:
\begin{enumerate}
 \item there exist two integers $1\leq i,j\leq n$ such that $s_i=0$ and $s_j=2$ ;
 \item for every pair of integers $(i,j)$ with $1\leq i,j\leq n$ such that $s_i=s_j=0$ and $s_k\neq0$ for all $k\in I_{n}(i,j)$, there exists a unique $k_0\in I_{n}(i,j)$ such that $s_{k_0}=2$.
\end{enumerate}

Note that an interlacing word takes the values $0$ and $2$ an equal number of times. We can now characterize the words that may arise as
the occupancy word of some collection of points $P_1,\ldots,P_n$. We call such words \emph{realizable}.

\begin{theorem}
\label{thm:main}
A word $\underline{u}=(u_1,\ldots,u_{2n})\in\{0,1\}^{2n}$ is realizable if and only if its signature is interlacing.
\end{theorem}

In order to get rid of the arbitrary choice of the position on the circle at which we start reading the word, as well as the choice of an orientation of the circle, we will consider \emph{bracelets}, which are equivalence classes of words considered up to cyclic shifts and reversal (see e.g. \cite{HP}). We denote by $\calB_n$ (resp.~$\calW_n$) the set of realizable bracelets (resp.~words) of length $2n$.

Theorem~\ref{thm:main} seems to be new even in the case $n=3$. We can state a finer version of this result in that case. If $A$ and $B$ are two points in the plane, we denote by $|AB|$ the Euclidean distance between $A$ and $B$.

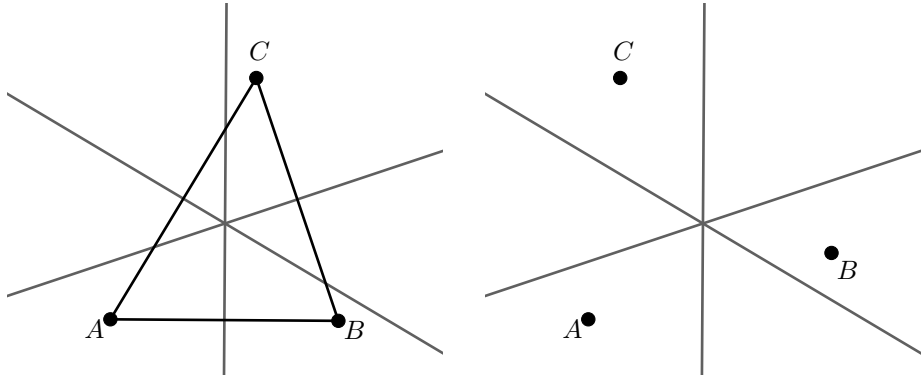
\begin{figure}[h]
  \centering
\definecolor{ttzzqq}{rgb}{0.2,0.6,0.}
\definecolor{zzttqq}{rgb}{0.6,0.2,0.}
\definecolor{wqwqwq}{rgb}{0.3764705882352941,0.3764705882352941,0.3764705882352941}
\begin{tikzpicture}[line cap=round,line join=round,>=triangle 45,x=1.0cm,y=1.0cm]
\clip(5.2276601546797,-2.589622005417789) rectangle (10.952017512688182,2.3248045889365323);
\draw [line width=1.pt,color=wqwqwq,domain=5.2276601546797:11.652017512688182] plot(\x,{(-24.2774--3.*\x)/0.02});
\draw [line width=1.pt,color=wqwqwq,domain=4.2276601546797:12.652017512688182] plot(\x,{(--10.6326-1.08*\x)/-3.22});
\draw [line width=1.pt,color=wqwqwq,domain=4.2276601546797:12.652017512688182] plot(\x,{(--13.6448-1.92*\x)/3.2});
\draw [line width=1.pt] (6.58,-1.86)-- (9.58,-1.88);
\draw [line width=1.pt] (9.58,-1.88)-- (8.5,1.34);
\draw [line width=1.pt] (8.5,1.34)-- (6.58,-1.86);

\draw [fill=black] (6.58,-1.86) circle (2.5pt);
\draw[color=black] (6.628183679918175,-1.739352532211187) node [below left] {$A$};
\draw [fill=black] (9.58,-1.88) circle (2.5pt);
\draw[color=black] (9.527229156489566,-1.7586596919316038) node [below right] {$B$};
\draw [fill=black] (8.5,1.34) circle (2.5pt);
\draw[color=black] (8.546028212146233,1.4592002614711828) node [above] {$C$};
\end{tikzpicture}
\definecolor{ttzzqq}{rgb}{0.2,0.6,0.}
\definecolor{zzttqq}{rgb}{0.6,0.2,0.}
\definecolor{wqwqwq}{rgb}{0.3764705882352941,0.3764705882352941,0.3764705882352941}
\begin{tikzpicture}[line cap=round,line join=round,>=triangle 45,x=1.0cm,y=1.0cm]
\clip(5.2276601546797,-2.589622005417789) rectangle (10.952017512688182,2.3248045889365323);
\draw [line width=1.pt,color=wqwqwq,domain=5.2276601546797:11.652017512688182] plot(\x,{(-24.2774--3.*\x)/0.02});
\draw [line width=1.pt,color=wqwqwq,domain=4.2276601546797:12.652017512688182] plot(\x,{(--10.6326-1.08*\x)/-3.22});
\draw [line width=1.pt,color=wqwqwq,domain=4.2276601546797:12.652017512688182] plot(\x,{(--13.6448-1.92*\x)/3.2});

\draw [fill=black] (6.58,-1.86) circle (2.5pt);
\draw[color=black] (6.628183679918175,-1.739352532211187) node [below left] {$A$};
\draw [fill=black] (9.78,-0.98) circle (2.5pt);
\draw[color=black] (9.727229156489566,-0.9586596919316038) node [below right] {$B$};
\draw [fill=black] (7.0,1.34) circle (2.5pt);
\draw[color=black] (7.046028212146233,1.4592002614711828) node [above] {$C$};
\end{tikzpicture}
  \caption{For $n=3$, there is only one realizable bracelet, shown on the left. The topological configuration
    on the right is impossible to achieve with the lines being
    the perpendicular bisectors of the segments.}
  \label{fig:triangle}
\end{figure}

\begin{proposition}
\label{prop:3}
Let $A,B,C$ be three points in the plane with $|AB|<|BC|<|CA|$. Then the three perpendicular bisectors of the triangle $ABC$ divide the plane into six regions satisfying the following properties:
\begin{itemize}
 \item $A$ and $B$ lie in two consecutive regions;
 \item the regions containing $B$ and $C$ are separated by one empty region;
 \item the regions containing $C$ and $A$ are separated by two empty regions.
\end{itemize}
In particular, the equivalence class of $(0,1,0,0,1,1)$ is the only realizable bracelet for $n=3$.
\end{proposition}

The proof is elementary and goes along the following lines: for each triple consisting of two points and a line, determine whether the line separates the points by using the inequalities $|AB|<|BC|<|CA|$. Note that there are only three bracelets composed of three $0$'s and three $1$'s, namely the equivalence classes of $(0,1,0,0,1,1)$, $(0,1,0,1,0,1)$ and $(0,0,0,1,1,1)$. The last bracelet is clearly not realizable, otherwise all three points would lie on one side of one of the perpendicular bisectors. Proposition~\ref{prop:3} implies that the bracelet of $(0,1,0,1,0,1)$ (see Figure~\ref{fig:triangle}, right) is not realizable. Figure~\ref{fig:triangle}, left, depicts a realization of $(0,1,0,0,1,1)$.

We now provide an enumerative result for realizable words and bracelets. It follows from Theorem~\ref{thm:main} and is to be compared with the total number of words of length $2n$ containing $n$ ones and $n$ zeros, which is known to be $\binom{2n}{n} = 4^{n(1+o(1))}$ by Stirling's formula.

\begin{corollary}
  \label{cor:enum}
  The number of realizable words is
  \begin{equation}
    \# \calW_n = 3^n - 2^{n+1}+1.
  \end{equation}
As a consequence, the exponential growth rate of the number of realizable bracelets is equal to $3$, that is, $\# \calB_n= 3^{n(1+o(1))}$.
\end{corollary}

The sequence $(\# \calW_n)_{n\geq1}$ is listed in the Online Encyclopedia of Integer Sequences~\cite{OEIS} (OEIS) under the number A028243 and corresponds to twice the Stirling numbers of second kind.

\begin{table}[htbp]
\centering
\begin{tabular}{|c|c|c|c|c|c|c|c|c|c|c|}
  \hline
  $n$ & 3 & 4 & 5 & 6 & 7 & 8 & 9 & 10 & 11 & 12 \\
  \hline
  $\# \calB_n$  & 1 & 5 & 9 & 30 & 69 & 203 & 519 & 1466 & 3933 & 11025 \\
  \hline
\end{tabular}
\caption{First terms of the sequence $(\# \calB_n)_{n\geq3}$ counting the number of realizable bracelets.}
\label{tab:firstterms}
\end{table}

The sequence $(\# \calB_n)_{n\geq3}$ was absent of the OEIS before our work, so we added it under the number A350280, computing the first few terms using a brute-force algorithm, see Table~\ref{tab:firstterms}. After this addition, OEIS editor A.~Howroyd was able to compute many more terms by finding an explicit formula for $\# \calB_n$, which follows from Theorem~\ref{thm:main}, Corollary~\ref{cor:enum} and Burnside's lemma. For completeness, we state this formula here without proof. Let $\phi$ denote Euler's totient function.

\begin{corollary}[{\cite[Sequence A350280]{OEIS}}]
  \label{cor:howroyd}
For $n \geq 3$, if $n$ is odd,
  \begin{equation}
    \# \calB_n = \frac{1}{4n} \sum_{d | n} \phi\left(\frac{n}{d}\right) \left(3^d - 2^{d+1} + 1 \right),
  \end{equation}
  and if $n$ is even,
  \begin{equation}
    \# \calB_n = \frac{3^{\frac{n}{2}-1} + 1 }{2} + \frac{1}{4n} \sum_{\substack{d | n, \\ \frac{n}{d} \text{ odd}}} \phi\left(\frac{n}{d}\right) \left(3^d - 2^{d+1} - 1 \right).
  \end{equation}
\end{corollary}

Theorem~\ref{thm:main} and Corollary~\ref{cor:enum} are proved in Section~\ref{sec:charac}.

\subsection{Uniformly random points on the circle}
\label{subsec:uniformpoints}

We now turn to the first of two natural models for sampling random realizable words or bracelets. Let $n\geq3$. Select $n$ i.i.d. points uniformly at random on the circle and consider the realizable word $\underline{u}$ obtained from it. For any bracelet $b\in\calB_n$, we denote by $\P(b)$ the probability of achieving $b$ via $n$ i.i.d. uniform random points.

\begin{proposition}
\label{prop:rational}
For every $n\geq3$ and $b\in\calB_n$, $\P(b)$ is a rational number.
\end{proposition}

The exact computation of $\P(b)$ for a fixed $b$ is possible but the technique we use requires to compute a number of integrals which is exponential in the number of occurrences of the letter $2$ in the signature of $b$. We provide such a computation in a special case. For every $n\geq3$, define the bracelet $b_n\in\calB_n$ to be the equivalence class of $(1,0,1,\ldots,1,0,\ldots,0)$, which is the word composed of a $1$, followed by a $0$, then $n-1$ $1$'s and finally $n-1$ $0$'s.

\begin{proposition}
\label{prop:nicebracelet}
For every $n\geq3$, we have
\[
\P(b_n)=\frac{n}{3\cdot 2^{2n-6}}.
\]
\end{proposition}

Beyond the probability of individual bracelets, we provide simple exact formulas for the expectations of some statistics. For every $k\in\{0,1,2\}$, a region in a configuration is said to be \emph{of type~$k$} if the total number of points contained in the union of the region with its antipodal region is $k$. In other words, a region of type~$k$ corresponds to a value of $k$ in the signature of the word describing the configuration.

For every $k\in\{0,1,2\}$ and $n\geq3$, denote by $H_{k,n}$ the random variable defined as the number of regions of type~$k$ in a configuration associated with $n$ i.i.d. uniform random points. It follows from the interlacing condition of Theorem~\ref{thm:main} that $H_{0,n}=H_{2,n}$. Since there are $2n$ regions in total we also have that $H_{1,n}=2n-2H_{2,n}$. We prove the following for the expectation of $H_{k,n}$, with $k\in\{0,1,2\}$:

\begin{theorem}
\label{thm:greenexpectation}
For every $n\geq3$, we have
\begin{align}
\E[H_{0,n}]=\E[H_{2,n}]&=\frac{n}{2}\left(1+\frac{1}{3^{n-2}}\right), \\
\E[H_{1,n}]&=n\left(1-\frac{1}{3^{n-2}}\right).
\end{align}
\end{theorem}

We also provide a formula for the expected total lengths of the arcs corresponding to regions of type~$k$. Here we rescale the distances on the circle so that the total length of the circle is $1$.

For every $k\in\{0,1,2\}$ and $n\geq3$, denote by $L_{k,n}$ the random variable defined as the sum of the lengths of all the regions of type~$k$.

\begin{theorem}
\label{thm:greenredlengths}
For every $n\geq3$, we have
\begin{align}
\E[L_{0,n}]&=\frac{3^{n-1}+2n-7}{8\cdot 3^{n-1}}, \\
\E[L_{1,n}]&=\frac{3^{n-1}-n-1}{2\cdot 3^{n-1}}, \\
\E[L_{2,n}]&=\frac{3^n+2n+11}{8\cdot 3^{n-1}}.
\end{align}
\end{theorem}

Each region either contains a single point or is empty. For $n$ points on the circle, there are exactly $n$ occupied regions and $n$ empty regions. Denote by $L_{e,n}$ the sum of the lengths of the empty regions. Since $L_{e,n}=L_{0,n}+\tfrac{L_{1,n}}{2}$, we deduce the following from Theorem~\ref{thm:greenredlengths}:

\begin{corollary}
\label{cor:emptinessproba}
For every $n\geq3$, we have
\[
\E[L_{e,n}]=\frac{3}{8}-\frac{1}{8\cdot 3^{n-3}}.
\]
\end{corollary}

The quantity $\E[L_{e,n}]$ can be interpreted as the probability of landing in an empty region when selecting a location uniformly at random on a circle with $n$ uniformly random points.

From the formulas for fixed $n$, we immediately deduce the first order asymptotic behavior as the number of points $n$ goes to infinity.

\begin{corollary}
\label{cor:asymptotics}
We have the following asymptotic results for $n$ uniform i.i.d. points on the circle in the limit when $n$ tends to infinity:
\begin{enumerate}
 \item the asymptotic fraction of the number of regions of type~$0$, $1$ and $2$ is respectively $\tfrac{1}{4}$, $\tfrac{1}{2}$ and $\tfrac{1}{4}$ ;
 \item the asymptotic fraction of the length covered by regions of type~$0$, $1$ and $2$ is respectively $\tfrac{1}{8}$, $\tfrac{1}{2}$ and $\tfrac{3}{8}$ ;
 \item the asymptotic fraction of the length covered by empty regions is $\tfrac{3}{8}$.
\end{enumerate}
\end{corollary}

For every $k\in\{0,1,2\}$ we show in addition that, for the model of uniform random points, the regions of type~$k$ are asymptotically equidistributed around the circle, when we consider either their cardinality or their total length. Roughly speaking, the numbers and total lengths of regions of each type in a given portion of the circle are asymptotically proportional to the size of this portion. Specifically, for every $k\in\{0,1,2\}$ and $t \in [0,1]$, we define $h_{k,n}(t)$ and $\ell_{k,n}(t)$ to be respectively the number of regions of type~$k$ and the sum of the lengths of regions of type~$k$ which are entirely contained in the arc from $1$ to $e^{2i\pi t}$. Note that $h_{k,n}(1)=H_{k,n}$ and $\ell_{k,n}(1)=L_{k,n}$.

Then, the following holds in the space $\D([0,1],\R^3)$ of càdlàg functions from $[0,1]$ to $\R^3$, endowed with the $J_1$ topology (we refer to \cite{Kal02} for more background on that topology).

\begin{theorem}
\label{thm:functionalcv}
The following convergences hold in distribution in the space $\D([0,1],\R^3)$:
\begin{itemize}
\item[(i)]
\begin{align*}
\left( \frac{h_{0,n}(t)}{2n}, \frac{h_{1,n}(t)}{2n}, \frac{h_{2,n}(t)}{2n} \right)_{0 \leq t \leq 1} \underset{n \rightarrow \infty}{\overset{(d)}{\longrightarrow}} \left( \frac{t}{4}, \frac{t}{2}, \frac{t}{4} \right)_{0 \leq t \leq 1}
\end{align*}
\item[(ii)]
\begin{align*}
\left( \ell_{0,n}(t), \ell_{1,n}(t), \ell_{2,n}(t) \right)_{0 \leq t \leq 1} \underset{n \rightarrow \infty}{\overset{(d)}{\longrightarrow}} \left( \frac{t}{8}, \frac{t}{2}, \frac{3t}{8} \right)_{0 \leq t \leq 1}.
\end{align*}
\end{itemize}
\end{theorem}

Propositions~\ref{prop:rational} and~\ref{prop:nicebracelet} as well as Theorems~\ref{thm:greenexpectation},~\ref{thm:greenredlengths} and~\ref{thm:functionalcv} are proved in Section~\ref{sec:uniform}.

\subsection{Uniformly random realizable configurations}
\label{subsec:uniformbracelets}

Another way to study what a typical realizable configuration looks like is to sample a word or bracelet uniformly at random among all realizable words or bracelets of a given size.

We prove some asymptotic results for the shape of a realizable word or bracelet sampled uniformly at random in $\calW_n$ or $\calB_n$. Let $\underline{w}^{(n)}$ be a random word taken uniformly in the set of
realizable words of length $2n$. For $x \in [0,n]$ a real number and $k \in \{0,1,2 \}$, denote by $F^{k}_x$ the random variable corresponding to the number of occurrences of the letter $k$ in the signature of $\underline{w}^{(n)}$ between positions $0$ and $\lfloor x \rfloor$. Then the following holds:

\begin{theorem}
\label{thm:convergence}
\begin{itemize}
\item[(i)]The following holds in probability:
\begin{align*}
\left( \frac{F^0_n}{n}, \frac{F^1_n}{n}, \frac{F^2_n}{n} \right) \underset{n \rightarrow \infty}{\longrightarrow} \left(\frac{1}{6}, \frac{2}{3}, \frac{1}{6} \right).
\end{align*}

\item[(ii)]
We have the following convergence in distribution :
\begin{align*}
&\frac{2}{\sqrt{n}}\left( F^0_{cn}-\frac{cn}{6}, F^1_{cn}-\frac{2cn}{3}, F^2_{cn}-\frac{cn}{6}\right)_{0 \leq c \leq 1} \\
&\underset{n \rightarrow \infty}{\overset{(d)}{\longrightarrow}} \left( W_c, -2W_c, W_c \right)_{0 \leq c \leq 1}
\end{align*}
where $W$ is a Brownian motion of variance $2/9$.
\end{itemize}
\end{theorem}

The first item of Theorem~\ref{thm:convergence} is a law of large numbers while the second item is a functional central limit theorem. Theorem~\ref{thm:convergence} also holds if we replace a uniformly random realizable word by a uniformly random realizable bracelet (see Remark~\ref{rem:braceletnecklace}).  Theorem~\ref{thm:convergence} is proved in Section~\ref{sec:uniformwords}. As stated, Theorem~\ref{thm:convergence} describes the shape of the signature of uniform random realizable word. In Section~\ref{sec:uniformwords} we actually prove a more refined result, Theorem~\ref{thm:refinedconvergence}, which describes the shape of the word itself.

\subsection{Discussion of the results and open questions}
\label{subsec:discussion}

To conclude this introduction we state a few remarks and open questions.
\begin{itemize}
 \item We have provided computational evidence for most of the results of this introduction. It is contained in two supplementary data files that can be found in the sources. The first file is a code file, written in SageMath (version 9.2) as a Jupyter notebook. The whole notebook takes about one hour to be executed on a standard laptop. The second file is an HTML page that cannot be executed but that allows one to directly visualize the output.
 \item Comparing Corollary~\ref{cor:enum} with Proposition~\ref{prop:nicebracelet}, it is clear that the probability distribution on realizable words or bracelets obtained by sampling $n$ i.i.d. uniform points on the circle differs from the uniform distribution whenever $n$ is large enough. Using the exact initial values of $\#\calB_n$ computed in \cite[Sequence A350280]{OEIS} and some simple inequalities, it is actually possible to show that these two probability distributions differ for every $n\geq4$.
 \item Note that for a large bracelet chosen uniformly at random, the asymptotic fraction of regions of type~$2$ is $\tfrac{1}{6}$ while for a large bracelet constructed from uniform random points on the circle, that fraction is $\tfrac14$.
 \item In the case of uniform random points on the circle, we did not manage to prove a central limit theorem in the vein of
Theorem~\ref{thm:convergence} $(ii)$, although we conjecture that such a statement should hold.
 \item All the formulas of probabilities and expectations in the model of $n$ uniform random points for fixed $n$ (namely those of Proposition~\ref{prop:nicebracelet}, Theorems~\ref{thm:greenexpectation} and~\ref{thm:greenredlengths} and Corollary~\ref{cor:emptinessproba}) are very compact, yet their proofs involve quite lengthy computations. It would be interesting to find shorter and more conceptual proofs of these formulas. 
 \item As a generalization of Theorem~\ref{thm:main}, it would be interesting to characterize the occupancy words arising when we drop the requirement for the points $P_i$ to be arranged in cyclic order. As shown on Figure~\ref{fig:nonconvex}, in that case the occupancy word may have letters greater than $1$.
\end{itemize}

\begin{figure}[htbp]
\centering
\includegraphics[width=1.5in]{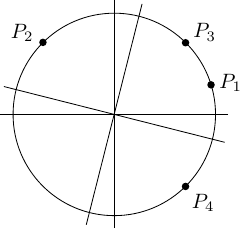}
\caption{An example with $n=4$ where the points $P_i$ are not in cyclic order. Here the occupancy word is $(2,0,1,0,0,0,1,0)$.}
\label{fig:nonconvex}
\end{figure}

\subsection*{Organization of the paper}
In Section~\ref{sec:charac} we prove Theorem~\ref{thm:main} about the characterization of realizable words and Corollary~\ref{cor:enum} about their enumeration. In Section~\ref{sec:uniform} we study the model of uniform random points on the circle and prove the results about this model presented in the introduction.
Finally in Section~\ref{sec:uniformwords} we prove Theorem~\ref{thm:convergence} describing a large realizable word chosen uniformly at random.

\section{Characterization and enumeration of realizable words}
\label{sec:charac}

In Subsection~\ref{subsec:nec}, we prove one direction of Theorem~\ref{thm:main}: the interlacement condition is necessary for a realizable word. The converse is proved in Subsection~\ref{subsec:suff}, using an explicit procedure to construct points from a given word with interlacing signature. Finally in Subsection~\ref{subsec:enumeration} we prove Corollary~\ref{cor:enum} about exact and asymptotic enumeration results.

\subsection{Necessary condition for a realizable word}
\label{subsec:nec}

The unit circle may be identified to the half-open interval $(0,1]$ via the inverse of the map $x\mapsto e^{2i\pi x}$. Denote by $p_1,\ldots,p_n$ the $n$ elements of $(0,1]$ corresponding to the $n$ points $P_1,\ldots,P_n$. For every $1\leq i \leq n$ define the midpoints
\[
l_i=\begin{cases}
\frac{p_i+p_{i+1}}{2} &\text{ if } p_i<p_{i+1} \\
\frac{1+p_i+p_{i+1}}{2} \mod 1 &\text{ if } p_i>p_{i+1} \\
\end{cases}
\]
where the representative modulo $1$ is taken to be in $(0,1]$ and the indices are considered modulo $n$.

 Up to applying a rotation of the circle, one may assume that $l_n=1$. Then we have
\[
0<p_1<l_1<p_2<l_2<\cdots<l_{n-1}<p_n<l_n=1.
\]
The inequalities are strict because of the genericity assumption.

Define also for every $1\leq i \leq n$, $p_i'=p_i+\tfrac{1}{2} \mod 1$ and $l_i'=l_i+\tfrac{1}{2} \mod 1$, the representatives being taken in $(0,1]$. Write $\calP=\{p_1,\ldots,p_n\}$, $\calP'=\{p'_1,\ldots,p'_n\}$, $\calL=\{l_1,\ldots,l_n\}$ and $\calL'=\{l'_1,\ldots,l'_n\}$.
Let $(m_i)_{1\leq i \leq 2n}$ be the reordering of the $l_i$ and $l_i'$, that is,
\[
\{m_i\}_{1\leq i \leq 2n}=\calL \cup \calL'
\]
and
\[
0<m_1<m_2<\cdots<m_{2n-1}<m_{2n}=1.
\]
Here again the inequalities are strict by the genericity assumption. We also set $m_0=0$. Similarly, let $(q_i)_{1\leq i \leq 2n}$ be the reordering of $\calP\cup\calP'$. For any $1\leq i \leq 2n$, we have
\[
v_i=\# [m_{i-1},m_i]  \cap \calP.
\]
Thus for any $1\leq i \leq n$, the signature $\underline{\sigma}$ of the occupancy word associated to $\calP$ satisfies
\[
\sigma_i=\# \left([m_{i-1},m_i] \cup [m_{i+n-1},m_{i+n}] \right) \cap \calP.
\]

Note that for every $1\leq i \leq n$,
\[
\sigma_i=\# [m_{i-1},m_i]\cap\left(\calP\cup\calP'\right).
\]

For any $(a,b)\in(0,1]^2$ define the circular distance
\[
d(a,b)=\min(|b-a|,1-|b-a|)
\]
to be the distance between $a$ and $b$ measured on the circle obtained
by identifying the two endpoints of the interval $[0,1]$.
We also introduce a notion of circular interval defined as follows. Let $a$ and $b$ be two elements of $(0,1]^2$ and define the circular interval
\begin{equation}
\label{eq:cyclic}
I(a,b)=\begin{cases}
(a,b) &\text{ if } a\leq b \\
(0,b)\cup (a,1] &\text{ if } a> b. \\
\end{cases}
\end{equation}
We similarly define $I[a,b)$ to be the circular counterpart for the half-open interval $[a,b)$.

In the remainder of this section, the indices of $q$ (resp.~$\sigma$) will be considered modulo $2n$ (resp.~$n$) and the real numbers of the form $q-\tfrac{1}{2}$ and $q+\tfrac{1}{2}$ should be understood as the representative in $(0,1]$ of an equivalence class modulo $1$.

Let $p\in\calP$. We define $C(p)$ to be the element $x'\in\calP'$
which minimizes $d(p,x')$. By the genericity assumption, $C(p)$ is uniquely defined. Similarly for any $p'\in\calP'$, we define $C(p')$ to be the element $x\in\calP$ which minimizes $d(p',x)$. For any $q \in \calP \cup \calP'$, when $C(q)$ belongs to $I(q,q+\tfrac{1}{2})$ (resp.~$I(q-\tfrac{1}{2},q)$), we say that $q$ \emph{looks} to its right (resp.~left) and we write it $D(q)=R$ (resp.~$D(q)=L$).

\begin{figure}[h]
  \centering
  \def\svgwidth{12cm}
  
  \begingroup%
  \makeatletter%
  \providecommand\color[2][]{%
    \errmessage{(Inkscape) Color is used for the text in Inkscape, but the package 'color.sty' is not loaded}%
    \renewcommand\color[2][]{}%
  }%
  \providecommand\transparent[1]{%
    \errmessage{(Inkscape) Transparency is used (non-zero) for the text in Inkscape, but the package 'transparent.sty' is not loaded}%
    \renewcommand\transparent[1]{}%
  }%
  \providecommand\rotatebox[2]{#2}%
  \newcommand*\fsize{\dimexpr\f@size pt\relax}%
  \newcommand*\lineheight[1]{\fontsize{\fsize}{#1\fsize}\selectfont}%
  \ifx\svgwidth\undefined%
    \setlength{\unitlength}{203.88858104bp}%
    \ifx\svgscale\undefined%
      \relax%
    \else%
      \setlength{\unitlength}{\unitlength * \real{\svgscale}}%
    \fi%
  \else%
    \setlength{\unitlength}{\svgwidth}%
  \fi%
  \global\let\svgwidth\undefined%
  \global\let\svgscale\undefined%
  \makeatother%
  \begin{picture}(1,0.22401233)%
    \lineheight{1}%
    \setlength\tabcolsep{0pt}%
    \put(0,0){\includegraphics[width=\unitlength,page=1]{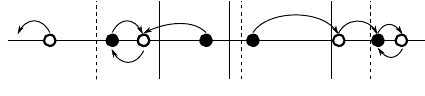}}%
    \put(0.11290264,0.19529557){\color[rgb]{0,0,0}\makebox(0,0)[t]{\lineheight{1.25}\smash{\begin{tabular}[t]{c}$L$\end{tabular}}}}%
    \put(0,0){\includegraphics[width=\unitlength,page=2]{neighbours.pdf}}%
    \put(0.26004181,0.19613066){\color[rgb]{0,0,0}\makebox(0,0)[t]{\lineheight{1.25}\smash{\begin{tabular}[t]{c}$R$\end{tabular}}}}%
    \put(0.33277633,0.19594589){\color[rgb]{0,0,0}\makebox(0,0)[t]{\lineheight{1.25}\smash{\begin{tabular}[t]{c}$L$\end{tabular}}}}%
    \put(0.48075064,0.19529557){\color[rgb]{0,0,0}\makebox(0,0)[t]{\lineheight{1.25}\smash{\begin{tabular}[t]{c}$L$\end{tabular}}}}%
    \put(0.59110503,0.19529557){\color[rgb]{0,0,0}\makebox(0,0)[t]{\lineheight{1.25}\smash{\begin{tabular}[t]{c}$R$\end{tabular}}}}%
    \put(0.79342135,0.19529557){\color[rgb]{0,0,0}\makebox(0,0)[t]{\lineheight{1.25}\smash{\begin{tabular}[t]{c}$R$\end{tabular}}}}%
    \put(0.88603359,0.19529557){\color[rgb]{0,0,0}\makebox(0,0)[t]{\lineheight{1.25}\smash{\begin{tabular}[t]{c}$R$\end{tabular}}}}%
    \put(0.94056057,0.19529557){\color[rgb]{0,0,0}\makebox(0,0)[t]{\lineheight{1.25}\smash{\begin{tabular}[t]{c}$L$\end{tabular}}}}%
    \put(0,0){\includegraphics[width=\unitlength,page=3]{neighbours.pdf}}%
    \put(0.66651386,0.19529567){\color[rgb]{0,0,0}\makebox(0,0)[t]{\lineheight{1.25}\smash{\begin{tabular}[t]{c}$R$\end{tabular}}}}%
    \put(0,0){\includegraphics[width=\unitlength,page=4]{neighbours.pdf}}%
    \put(0.03933303,0.19529557){\color[rgb]{0,0,0}\makebox(0,0)[t]{\lineheight{1.25}\smash{\begin{tabular}[t]{c}$L$\end{tabular}}}}%
    \put(0,0){\includegraphics[width=\unitlength,page=5]{neighbours.pdf}}%
    \put(0.11014379,0.0058535){\color[rgb]{0,0,0}\makebox(0,0)[t]{\lineheight{1.25}\smash{\begin{tabular}[t]{c}$1$\end{tabular}}}}%
    \put(0.03197612,0.00539368){\color[rgb]{0,0,0}\makebox(0,0)[t]{\lineheight{1.25}\smash{\begin{tabular}[t]{c}$1$\end{tabular}}}}%
    \put(0.1906105,0.00539368){\color[rgb]{0,0,0}\makebox(0,0)[t]{\lineheight{1.25}\smash{\begin{tabular}[t]{c}$0$\end{tabular}}}}%
    \put(0.29912568,0.0058535){\color[rgb]{0,0,0}\makebox(0,0)[t]{\lineheight{1.25}\smash{\begin{tabular}[t]{c}$2$\end{tabular}}}}%
    \put(0.46511697,0.00631331){\color[rgb]{0,0,0}\makebox(0,0)[t]{\lineheight{1.25}\smash{\begin{tabular}[t]{c}$1$\end{tabular}}}}%
    \put(0.55478002,0.0058535){\color[rgb]{0,0,0}\makebox(0,0)[t]{\lineheight{1.25}\smash{\begin{tabular}[t]{c}$0$\end{tabular}}}}%
    \put(0.60030116,0.00539368){\color[rgb]{0,0,0}\makebox(0,0)[t]{\lineheight{1.25}\smash{\begin{tabular}[t]{c}$1$\end{tabular}}}}%
    \put(0.70145932,0.00539357){\color[rgb]{0,0,0}\makebox(0,0)[t]{\lineheight{1.25}\smash{\begin{tabular}[t]{c}$1$\end{tabular}}}}%
    \put(0.82330892,0.0058535){\color[rgb]{0,0,0}\makebox(0,0)[t]{\lineheight{1.25}\smash{\begin{tabular}[t]{c}$1$\end{tabular}}}}%
    \put(0.9221681,0.00539368){\color[rgb]{0,0,0}\makebox(0,0)[t]{\lineheight{1.25}\smash{\begin{tabular}[t]{c}$2$\end{tabular}}}}%
  \end{picture}%
\endgroup%

  \caption{A configuration on a portion of $(0,1]$. The black
    (resp.~white) dots represent elements of $\calP$ (resp.~$\calP'$),
    and the vertical solid (resp.~dashed) lines represent elements of
    $\calL$ (resp.~$\calL'$). From each dot $q$, the arrow is directed
    towards
    $C(q)$, and above it, the value of $D(q)$ is written. Below each
    region, the corresponding letter in $\underline{\sigma}$ is
    indicated.}
  \label{fig:ex_conf}
\end{figure}

Our aim in this subsection is to prove the following.
\begin{proposition}
  \label{prop:nec}
Let $\underline{u}$ be the occupancy word associated to a collection of points in cyclic order. Then the signature $\underline{\sigma}$ of $\underline{u}$ is interlacing.
\end{proposition}

A key observation in this article is that the occurrences of $2$ (resp.~$0$) in
$\underline{\sigma}$ exactly correspond to the occurrences of the pattern $RL$ (resp.~$LR$) in the successive
values of $\left( D(q_i) \right)_{i=1,\dots,2n}$, observation which is made explicit in Lemmas~\ref{lem:match} and~\ref{lem:green} (resp.~in Lemma~\ref{lem:red}).

\begin{lemma}
\label{lem:match}
Let $1\leq i\leq 2n$ and assume that $D(q_i)=R$ and $D(q_{i+1}) = L$. Then exactly one of $q_i$ and $q_{i+1}$ belongs to $\calP$, and we have $C(q_i)=q_{i+1}$ and $C(q_{i+1})=q_i$.
\end{lemma}

\begin{proof}
Up to performing a rotation of the circle, one may assume that $q_i<q_{i+1}$  (this is needed to take into account the case $i=2n$). Reason by contradiction and assume that both $q_i$ and $q_{i+1}$ are in $\calP$. Since $q_{i+1}\in\calP$, we cannot have $C(q_i)=q_{i+1}$, and it follows from the fact that $I(q_i,q_{i+1})\cap\calP'=\varnothing$  that $C(q_i)\in I(q_{i+1},q_{i+1}+\tfrac{1}{2})$. The element of $\calP'$ in $I(q_{i+1},q_{i+1}+\tfrac{1}{2})$ which is closest to $q_i$ is also the closest to $q_{i+1}$, hence $C(q_{i+1})=C(q_i)$, so that $C(q_{i+1})$ belongs to $ I(q_{i+1}-\tfrac{1}{2},q_{i+1}) \cap I(q_{i+1},q_{i+1}+\tfrac{1}{2})= \varnothing$, which leads to a contradiction. Similarly, $q_i$ and $q_{i+1}$ cannot both be in $\calP'$. The last two statements of the lemma follow from the fact that $I(q_i,q_{i+1})\cap\left(\calP\cup\calP'\right)=\varnothing$.
\end{proof}

\begin{lemma}
\label{lem:green}
Let $p\in\calP$ and $x'\in\calP'$ be such that $x'=C(p)$ and $p=C(x')$. Let also $1\leq i \leq 2n$ be such that $p\in[m_{i-1},m_i]$. Then $\sigma_i=2$.
Conversely, let $1\leq i \leq 2n$ be such that $\sigma_i=2$ and denote by $p\in\calP$ and $x'\in\calP'$ the two elements of $[m_{i-1},m_i]\cap\left(\calP\cup\calP'\right)$. Then $x'=C(p)$ and $p=C(x')$.
\end{lemma}

\begin{proof}
Assume that $x'=C(p)$, $p=C(x')$ and take $i$ such that $p\in[m_{i-1},m_i]$. The point $p$ is the element of $\calP$ closest to $x'$, hence no element of $\calL$ separates $x'$ from $p$. Similarly, since $x'$ is the element of $\calP'$ closest to $p$, no element of $\calL'$ separates $p$ from $x'$. Thus $x'\in[m_{i-1},m_i]$ and $\sigma_i=2$.

Conversely, let $1\leq i \leq 2n$ be such that $\sigma_i=2$ and denote by $p\in\calP$ and $x'\in\calP'$ the two elements of $[m_{i-1},m_{i}]\cap\left(\calP\cup\calP'\right)$. If we had $C(x')\neq p$, then the perpendicular bisector of $C(x')$ and $p$
would separate $x'$ from $p$, which is not the case. So $C(x')=p$ and similarly $C(p)=x'$.
\end{proof}

\begin{lemma}
  \label{lem:red}
  Let $1\leq i \leq 2n$ be such that $D(q_i)=L$ and $D(q_{i+1})=R$. Then
  there exists a unique $1\leq j\leq 2n$ such that $m_j$ and $m_{j+1}$ are both in
  $I(q_i,q_{i+1})$, and this $j$ satisfies
  $\sigma_{j+1}=0$.
  Conversely, assume that $1\leq j\leq 2n$ is such that
  $\sigma_{j+1}=0$. Denote by $q_i$ the largest element of $\calP \cup \calP'$
  smaller than $m_j$. Then $D(q_i)=L$ and
  $D(q_{i+1})=R$.
\end{lemma}

\begin{proof}
  Let $1\leq i \leq 2n$ be such that $D(q_i)=L$ and $D(q_{i+1})=R$. We distinguish three cases.

  The first case is when $q_i$ and $q_{i+1}$ are of different types, that is, one belongs to $\calP$ and the other to $\calP'$. By symmetry we may assume that $q_i\in\calP'$ and $q_{i+1}\in\calP$. An example can be seen around the leftmost empty region in Figure~\ref{fig:ex_conf}. Since $D(q_i)=L$, we have that $C(q_i)$ and $q_{i+1}$ are two consecutive elements in $\calP$. Hence
  $M :=\tfrac{C(q_i)+q_{i+1}}{2}\in\calL$. Since $q_i$ is closer to
  $C(q_i)$ than to $q_{i+1}$, we have that $M\in I(q_i,q_{i+1})$ and
  $M$ is the only element of $\calL$ in $I(q_i,q_{i+1})$. A similar
  argument shows that $M' :=\tfrac{q_i+C(q_{i+1})}{2}$ is the only
  element of $\calL'$ in $I(q_i,q_{i+1})$. Hence $I(q_i,q_{i+1})$
  contains exactly two elements of $\calL\cup\calL'$. Denoting them by
  $m_j$ and $m_{j+1}$, we conclude that $\sigma_{j+1}=0$.

 The second case is when $q_i$ and $q_{i+1}$ both belong to $\calP$ (see for example the configuration around the second $0$ in Figure~\ref{fig:ex_conf}). Then $\tfrac{q_i+q_{i+1}}{2}$ is the only element of $I(q_i,q_{i+1})\cap \calL$. Furthermore, $C(q_i)$ and $C(q_{i+1})$ are consecutive elements in $\calP'$, so $M'' :=\tfrac{C(q_i)+C(q_{i+1})}{2}\in\calL'$. Since $q_i$ is closer to $C(q_i)$ than to $C(q_{i+1})$, we have that $M''\in I(q_i, q_i+\tfrac{1}{2})$. Since $q_{i+1}$ is closer to $C(q_{i+1})$ than to $C(q_i)$, we have that $M''\in I(q_{i+1}-\tfrac{1}{2},q_{i+1})$. So $M''$ is the only element of $I(q_i,q_{i+1})\cap \calL'$. The conclusion follows as in the first case.

  The third case, when $q_i$ and $q_{i+1}$ both belong to $\calP'$, is treated like the second case.

  Conversely, assume $1\leq j\leq 2n$ is such that $\sigma_{j+1}=0$. Since two consecutive elements of $\calL$ (resp.~of $\calL'$) must be separated by an element of $\calP$ (resp.~of $\calP'$), we deduce that among $m_j$ and $m_{j+1}$, one belongs to $\calL$ and the other to $\calL'$. Denote by $q_i$ the largest element of $\calP \cup \calP'$ smaller than $m_j$. Then $q_{i+1}$ is bigger than $m_{j+1}$. If $q_i\in \calP$, consider the unique element of $\calL'\cap \{m_j,m_{j+1}\}$. It is the bisector of two points of $\calP'$ and these points cannot be in $I(q_i,q_{i+1})$. Moreover, $q_i$ is to the left of the bisector. This implies that $D(q_i)=L$. This works also in the case $q_i\in\calP'$. Similarly, one shows that $D(q_{i+1})=R$.
\end{proof}

\begin{remark}
\label{rem:signaturearrows}
In particular, Lemmas~\ref{lem:match},~\ref{lem:green} and~\ref{lem:red} imply that the signature of a realizable word uniquely determines the sequence $D(q_i)_{1\leq i\leq 2n}$.
\end{remark}

We now prove that at least one region of a realizable word is of type~$2$.

\begin{lemma}
\label{lem:existgreen}
There exists $1\leq i \leq n$ such that $\sigma_i=2$.
\end{lemma}

\begin{proof}
Consider a pair $(p,x')$ achieving the minimum
\[
\min_{\substack{q\in\calP \\ y'\in\calP'}} d(q,y').
\]

Let $1\leq i \leq 2n$ be such that $p\in[m_i,m_{i+1}]$. Since $C(p)=x'$ and $C(x')=p$, we deduce from Lemma~\ref{lem:green} that $\sigma_i=2$.
\end{proof}

The next lemma finally shows that, between two regions of type~$0$, there is always a region of type~$2$.

\begin{lemma}
\label{lem:interlacing}
Assume that $\sigma_1=0$ and that there exists $2\leq i \leq n$ such that $\sigma_i=0$. Then there exists $2\leq j \leq i-1$ such that $\sigma_j=2$.
\end{lemma}

\begin{proof}
  Since $\sigma_1=0$, we have $q_1>m_1$ and by Lemma~\ref{lem:red} we
  have that $D(q_1)=R$. Denote by $q_r$ the largest element of $\calP \cup \calP'$
  smaller than $m_{i-1}$. By Lemma~\ref{lem:red} we have that
  $D(q_r)=L$. Denote by $k$ the smallest integer such that
  $D(q_k)=L$. We have that $2 \leq k \leq r$. Furthermore,
  $D(q_{k-1})=R$, hence by Lemma~\ref{lem:match} we have that
  $C(q_{k-1})=q_k$ and $C(q_k)=q_{k-1}$. Let $j$ be such that
  $q_k\in[m_{j-1},m_j]$. Clearly $2\leq j \leq i-1$ and by Lemma
 ~\ref{lem:green} we have that $\sigma_j=2$.
\end{proof}

We now have all the tools to prove Proposition \ref{prop:nec}:

\begin{proof}[Proof of Proposition~\ref{prop:nec}]
  Let $P_1,\ldots,P_n$ be $n$ points in cyclic order on the circle and let $\underline{\sigma} := (\sigma_1,\ldots,\sigma_{n})\in\{0,1,2\}^{n}$ be the signature of their occupancy word.
Define $s_0,s_1$ and $s_2$ to be respectively the number of occurrences of the values $0$, $1$ and $2$ in the signature. Then $n=s_0+s_1+s_2$ and since there are $n$ points, we also have
\[
n=\sum_{i=1}^{n} \sigma_i=s_1+2s_2.
\]
Combining these two equations we obtain that $s_0=s_2$.  From Lemma~\ref{lem:existgreen} we get that $s_2\geq1$. Therefore $s_0\geq 1$. Assume that $1\leq i<j \leq n$ are such
that $\sigma_i=\sigma_j=0$ and $\sigma_k>0$ for all $i < k < j$. Up to applying a translation, one may assume that $i=1$. By
Lemma~\ref{lem:interlacing} we deduce the existence of some $k$ such that $1 < k < j$ and $\sigma_k=2$. Given that $s_0=s_2$, such a $k$ is necessarily unique. Hence we conclude that $\underline{\sigma}$ is interlacing.
\end{proof}

\subsection{Realizing a word with interlacing signature}
\label{subsec:suff}

In this subsection we construct an explicit configuration of points
from a word whose signature is interlacing.
\begin{proposition}
  \label{prop:suff}
  Let $n\geq3$ and let $\underline{v}=(v_1,\ldots,v_{2n})\in\{0,1\}^{2n}$ be such that its signature $\underline{\sigma}=(\sigma_1,\ldots,\sigma_{n})\in\{0,1,2\}^{n}$ is interlacing. Then there exists a configuration of points on the circle having $\underline{v}$ as an occupancy word.
\end{proposition}

\begin{proof}
We fix $n\geq 3$ and such a word $\underline{v}$. Up to applying a rotation one may assume that $\sigma_1=0$.

Denote by $T$ (resp.~$Z$) the subset of all $1\leq i \leq 2n$ such that $\sigma_i=2$ (resp.~$\sigma_i=0$). Here again the indices of $\sigma$ are considered modulo $n$. $T$ and $Z$ are respectively the locations of twos and zeros. The set $\{1,\ldots,2n\}\setminus(T\cup Z)$ is composed of several connected components, which are the intervals of integers between two consecutive elements of $T\cup Z$ (note that some of these intervals may be empty). We call such a connected component an \emph{ascending
  component} (resp.~a \emph{descending component}) if it is of the form $I_{2n}(i,j)$ with $i\in Z$ and $j\in T$ (resp.~$i\in T$ and
$j\in Z$). In particular, for all $k\in I_{2n}(i,j)$ we have $\sigma_k=1$. Defining $s := \# T = \# Z$, we let $i_1<\cdots<i_s$ be the ordering of $T$, and $j_1<\cdots<j_s$ be the ordering of $Z$. In particular, $j_1=1$.

To each $1\leq i \leq 2n$, we associate a position $r_i$ in $(0,1]$, which in the end of the process will be the position of a point (resp.~the antipode of a point) if $v_i=1$ (resp.~if $v_i=0$). The idea is to guarantee that for every $h$ in a descending component $I_{2n}(i_k,j_{k+1})$ (resp.~in an ascending component $I_{2n}(j_k,i_k)$), the position $r_h$ is closer to $r_{i_k}$ than to its closest neighbor on the right (resp.~left). Using the terminology of Subsection~\ref{subsec:nec}, we make sure that every point or antipode of a point in a descending (resp.~an ascending) component looks to the left (resp.~right). We will then check explicitly that the configuration of points thus constructed has occupancy word $\underline{v}$.

First, for all $1\leq k \leq s$, we set
\begin{equation*}
  \begin{split}
    r_{i_k} & = \frac{k}{s},\\
    r_{j_k} & = \frac{r_{i_{k-1}}+r_{i_k}}{2} = \frac{2k-1}{2s}.
  \end{split}
\end{equation*}
In the definition of $r_{j_1}$ we used the notational convention $r_{i_0}=0$. Since $s$ is even, the $2k$ points constructed arise in antipodal pairs, but this absence of genericity will not be an issue. In the last paragraph of the proof we will explain how one can perturb the positions to make them generic without changing the occupancy word.

Let $\eta>0$ be small enough ($\eta<\tfrac{1}{s2^{n+2}}$ will suffice
for our purposes). For $1\leq k \leq s$, consider the $k$th descending component, that is, $I_{2n}(i_k,j_{k+1})$. For every
$h\in I_{2n}(i_k,j_{k+1})$, we set
\begin{equation*}
  r_h = r_{i_k} + \eta \left( 2^{h-i_k} -1 \right) = \frac{k}{s} +  \eta \left( 2^{h-i_k} -1 \right) .
\end{equation*}
Similarly, for $1\leq k \leq s$, consider the $k$th ascending component, that is, $I_{2n}(j_k,i_{k})$. For every $h\in I_{2n}(j_k,i_{k})$, we set
\begin{equation*}
  r_h = r_{i_k} - \eta \left( 2^{i_k-h} -1 \right) = \frac{k}{s} -
  \eta \left( 2^{i_k-h} -1 \right).
\end{equation*}

\begin{figure}[h]
  \centering
  \def\svgwidth{10cm}
  
  \begingroup%
  \makeatletter%
  \providecommand\color[2][]{%
    \errmessage{(Inkscape) Color is used for the text in Inkscape, but the package 'color.sty' is not loaded}%
    \renewcommand\color[2][]{}%
  }%
  \providecommand\transparent[1]{%
    \errmessage{(Inkscape) Transparency is used (non-zero) for the text in Inkscape, but the package 'transparent.sty' is not loaded}%
    \renewcommand\transparent[1]{}%
  }%
  \providecommand\rotatebox[2]{#2}%
  \newcommand*\fsize{\dimexpr\f@size pt\relax}%
  \newcommand*\lineheight[1]{\fontsize{\fsize}{#1\fsize}\selectfont}%
  \ifx\svgwidth\undefined%
    \setlength{\unitlength}{182.96721691bp}%
    \ifx\svgscale\undefined%
      \relax%
    \else%
      \setlength{\unitlength}{\unitlength * \real{\svgscale}}%
    \fi%
  \else%
    \setlength{\unitlength}{\svgwidth}%
  \fi%
  \global\let\svgwidth\undefined%
  \global\let\svgscale\undefined%
  \makeatother%
  \begin{picture}(1,0.1883016)%
    \lineheight{1}%
    \setlength\tabcolsep{0pt}%
    \put(0,0){\includegraphics[width=\unitlength,page=1]{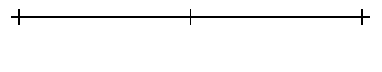}}%
    \put(0.04151218,0.17729336){\color[rgb]{0,0,0}\makebox(0,0)[t]{\lineheight{1.25}\smash{\begin{tabular}[t]{c}$r_{i_{k-1}}$\end{tabular}}}}%
    \put(0.96428466,0.17946716){\color[rgb]{0,0,0}\makebox(0,0)[t]{\lineheight{1.25}\smash{\begin{tabular}[t]{c}$r_{i_{k}}$\end{tabular}}}}%
    \put(0.50006211,0.1787426){\color[rgb]{0,0,0}\makebox(0,0)[t]{\lineheight{1.25}\smash{\begin{tabular}[t]{c}$r_{j_{k}}$\end{tabular}}}}%
    \put(0,0){\includegraphics[width=\unitlength,page=2]{construction.pdf}}%
    \put(0.05695666,0.09086436){\color[rgb]{0,0,0}\makebox(0,0)[t]{\lineheight{1.25}\smash{\begin{tabular}[t]{c}$\eta$\end{tabular}}}}%
    \put(0.09087335,0.09090585){\color[rgb]{0,0,0}\makebox(0,0)[t]{\lineheight{1.25}\smash{\begin{tabular}[t]{c}$2\eta$\end{tabular}}}}%
    \put(0.14816659,0.09089958){\color[rgb]{0,0,0}\makebox(0,0)[t]{\lineheight{1.25}\smash{\begin{tabular}[t]{c}$4\eta$\end{tabular}}}}%
    \put(0.90580677,0.0924948){\color[rgb]{0,0,0}\makebox(0,0)[t]{\lineheight{1.25}\smash{\begin{tabular}[t]{c}$2\eta$\end{tabular}}}}%
    \put(0.93972328,0.09253629){\color[rgb]{0,0,0}\makebox(0,0)[t]{\lineheight{1.25}\smash{\begin{tabular}[t]{c}$\eta$\end{tabular}}}}%
    \put(0,0){\includegraphics[width=\unitlength,page=3]{construction.pdf}}%
    \put(0.05695666,0.09086436){\color[rgb]{0,0,0}\makebox(0,0)[t]{\lineheight{1.25}\smash{\begin{tabular}[t]{c}$\eta$\end{tabular}}}}%
    \put(0.09087335,0.09090585){\color[rgb]{0,0,0}\makebox(0,0)[t]{\lineheight{1.25}\smash{\begin{tabular}[t]{c}$2\eta$\end{tabular}}}}%
    \put(0.49812175,0.00170897){\color[rgb]{0,0,0}\makebox(0,0)[t]{\lineheight{1.25}\smash{\begin{tabular}[t]{c}$1/s$\end{tabular}}}}%
  \end{picture}%
\endgroup%

  \caption{Schematic construction of the values of $r_h$, on a descending component $I_{2n}(i_{k-1},j_k)$ and an ascending component $I_{2n}(j_k,i_k)$. The multiples of $\eta$ written below are distances.}
\end{figure}

By the symmetry of the word, there cannot be more than $n$ points in a component. Therefore, in both of those cases,
\begin{equation}
  \label{eq:distmax}
  d(r_h,r_{i_k}) < \eta (2^n-1) <\frac{1}{4s}.
\end{equation}
Hence the constructed positions of ascending and descending components lie in disjoint intervals.

Now let $\calP$ be the subset of positions $\left\{ r_i \mid 1\leq i
  \leq 2n, v_i=1\right\}$. We claim that this configuration of points has
$\underline{v}$ as its occupancy word. As in the previous subsection, we set
$\calP' = \{p+\tfrac12, p \in \calP\}$, and $\calL$ (resp.~$\calL'$)
the positions of the bisectors of $\calP$ (resp.~$\calP'$). We also
set $\calM$ to be the collection of $\calL \cup \calL'$, possibly with
repetitions, since the configuration constructed may not yet satisfy the genericity assumptions.

We now characterize the positions of these bisectors.

\begin{lemma}
  \label{lemma:m}
  For every $h$ in a descending (resp.~ascending) component, there is a unique element of
  $\calM$ in $I(r_{h-1},r_h)$ (resp.~in $I(r_h,r_{h+1})$). For every $1\leq k \leq s$, there are exactly two elements of~$\calM$ in the set $I\left( \frac{2k-1}{2s} - \frac{1}{8s}, \frac{2k-1}{2s} + \frac{1}{8s} \right)$. Moreover, these are all the
  $2n$ elements of $\calM$.
\end{lemma}

\begin{proof}[Proof of Lemma~\ref{lemma:m}]
  Let $1\leq k \leq s$. Let $h$ be in the descending component
  $I_{2n}(i_k,j_{k+1})$. We distinguish two cases, depending on the
  value of $v_h$.

  If $v_h=1$, then $r_h\in\calP$. Moreover, $v_{i_k}=1$ since $i_k\in
  T$, hence $r_{i_k}\in \calP$. Therefore, the rightmost element of
  $\calP$ smaller than $r_h$ belongs to the set $\{r_{i_k},r_{i_k
    +1},\dots,r_{h-1}\}$. Hence the position $m_h$ of the bisector of this point and
  $r_{h}$ satisfies
  \begin{equation*}
    \frac{r_{i_k}+r_h}{2} \leq m_h \leq \frac{r_{h-1}+r_h}{2} < r_h.
  \end{equation*}
  Now notice that the left-hand side is $r_{i_k} + \eta
  (2^{h-i_k-1}-\tfrac12)$, which is strictly bigger than
  $r_{h-1}$. Hence $m_h \in I(r_{h-1},r_h)$.

  If $v_h=0$, then as $\sigma_h=1$, we have $v_{h+n}=1$. Hence there
  is an element of $\calP$ at position $r_{h+n}$, and by the
  invariance under translation of the word by $n$, we have
  $r_{h+n}=r_h+\tfrac12$. Therefore, $r_h\in \calP'$. Similarly, as
  $\sigma_{i_k}=2$, we have $v_{i_k+n}=1$ so that $r_{i_k+n}\in \calP$ and
  $r_{i_k}\in \calP'$. From there, we conclude as in the previous case.

  For $h$ in an ascending component, the proof is identical.

  For the second point of the lemma, consider the index $j_k \in Z$.

 As both $r_{i_{k-1}}$ and $r_{i_k}$ belong to $\calP$, the elements of $\calP$ directly to the left and right of $r_{j_k}$ belong, respectively, to $\{r_{i_{k-1}}, \dots, r_{j_k -1}\}$ and to $\{r_{j_{k}+1}, \dots,
  r_{i_k}\}$. Hence the position of their bisector $m_{j_k}\in \calL$ satisfies
  \begin{equation*}
    \frac{r_{i_{k-1}}+r_{j_{k}+1}}{2} \leq m_{j_k} \leq \frac{r_{j_k-1}+r_{i_k}}{2}.
  \end{equation*}
  As $j_k-1$ belongs to the descending component
  $I_{2n}(i_{k-1},j_k)$, by~\eqref{eq:distmax} we have $r_{j_{k}-1} < r_{i_{k-1}} + \frac{1}{4s}$, hence the right-hand side is smaller than $\frac{r_{i_{k-1}} + r_{i_{k}}}{2} + \frac{1}{8s}$, which is the expected bound. The left-hand side is treated similarly. Then, an identical proof shows that there is an element of $\calL'$ in the same interval.

  Clearly the elements of $\calM$ coming from ascending and descending components are disjoint. Those coming from the second point are also disjoint among themselves, as even if two may share the same position, only one of them will belong to $\calL$, and the other to
  $\calL'$. The fact that these two families are disjoint is an easy consequence of~\eqref{eq:distmax}. Hence we have constructed two elements of $\calM$ for each element of $Z$, and one for each element of $(T\cup Z)^c$, which is in total $2 (\# Z) + 2n - (\# T + \# Z) = 2n$.
\end{proof}

Let $\underline{w}$ be the occupancy word of $\calP$. We now have all
the tools to prove that $\underline{w}=\underline{v}$. For any $1\leq
k \leq s$, consider the interval $I[r_{i_{k-1}},r_{i_k})$. In that part,
the positions and order of the elements of $\calM$ are given by
Lemma~\ref{lemma:m}: for every $h$ in the descending component
$I_{2n}(i_{k-1},j_k)$ there is one $m_h \in I(r_{h-1},r_h) \cap \calM$; then there
are two distinct elements in $m_{j_k},m_{j_k}' \in I(r_{j_k-1},r_{j_k+1}) \cap \calM$; then for every $h$ in
the ascending component $I_{2n}(j_k,i_k)$ there is one $m_h \in
I(r_h,r_{h+1})$. Thus the part of $\underline{w}$ corresponding to
this interval can be described as: first a $1$ (for the region
containing $r_{i_{k-1}}$), then for every $h\in I_{2n}(i_{k-1},j_k)$,
either a $1$ or
a $0$ according to the value of $v_h$ (as by definition these are the
positions where a point of $\calP$ has been put), then a $0$ (for the
region corresponding to $m_{j_k},m_{j_k}'$), then for every $h \in
I_{2n}(j_k,i_k)$ either a $0$ or a $1$ according to the value of
$v_h$. This is clearly the same as $\underline{v}$ at those
indices. This being true for any $k$, by concatenation we get that $\underline{w}=\underline{v}$.

One issue that may arise is that the configuration constructed above
is not generic, in the sense that two lines may coincide, which occurs
for example when a descending component and the following ascending
component are empty. In order to avoid such issues, we slightly perturb the
configuration by fixing $\epsilon>0$ and defining for every $1\leq k\leq 2n$, the position $\tilde{r}_k=r_k+k\epsilon$. For
$\epsilon$ small enough ($\epsilon < \frac{\eta}{2n}$ suffices), the relative position of the perturbed points and lines is the same as the unperturbed one, while two lines can no longer coincide.
\end{proof}

\subsection{Enumerating realizable words and bracelets}
\label{subsec:enumeration}

We end this section by computing the cardinality of the set $\# \calW_n$.

\begin{proof}[Proof of Corollary~\ref{cor:enum}]
  To choose a realizable word $\underline{v}$, one may first choose its interlacing signature $\underline{\sigma}$. This amounts to choosing an integer $1\leq \ell\leq \lfloor \frac{n}{2} \rfloor$ such that $\underline{\sigma}$ will have $2\ell$ letters $0$ or $2$ and choose whether the first one is a $0$ or a $2$. We have $2\binom{n}{2\ell}$ choices for the positions of these letters and the value of the first one. Then for every $1\leq i\leq 2n$ such that $\sigma_i = 1$ (here the index $i$ is considered modulo $2n$), one has to choose if $v_i$ is $1$ or $0$, under the condition that $v_{i+n}\neq v_i$. This gives $2^{n-2\ell}$ choices. Hence the number of realizable words of size $2n$ is
  \begin{align}
      \# \calW_n & = 2 \sum_{\ell=1}^{\lfloor \frac{n}{2} \rfloor} \binom{n}{2\ell}
                2^{n-2\ell} \\
          & = 2 \sum_{\substack{2\leq k \leq n \\ k \text{ even}}} \binom{n}{k}
                2^{n-k}.
  \end{align}
  Introducing
  \[
  X = 2 \sum_{\substack{1\leq k \leq n \\ k \text{ odd}}}
  \binom{n}{k} 2^{n-k},
  \]
  one easily gets $2^{n+1}+\# \calW_n+X = 2 \times
  3^n$ and $2^{n+1} + \# \calW_n - X = 2$, and the result follows.

Regarding realizable bracelets, we have
\[
\frac{\# \calW_n}{4n}\leq \# \calB_n \leq \# \calW_n,
\]
which implies the
  announced result.
\end{proof}

\section{Bracelets for uniformly random points}
\label{sec:uniform}

In this section, we study the following model: fix $n\geq3$ and draw at random $n$ points independently and uniformly distributed on the unit circle. Most of our proofs here rely on a model of black and white dots with exponential spacings, which can be coupled to our original model of uniform points on the circle. This new model is presented in Subsection~\ref{ssec:bw}. We prove Proposition~\ref{prop:rational} and Proposition~\ref{prop:nicebracelet} concerning the probability of individual bracelets to occur in Subsection~\ref{ssec:individual}, then Theorem~\ref{thm:greenexpectation} about the expected number of regions of each type in Subsection~\ref{ssec:greenexpectation},  Theorem~\ref{thm:greenredlengths} about the expected total length of regions of type~$k$ for every $k\in\{0,1,2\}$ in Subsection~\ref{ssec:regionlengths} and finally Theorem~\ref{thm:functionalcv} about the equidistribution of the regions of type~$k$ in Subsection~\ref{ssec:equidistributed}.

\subsection{Black and white dots with exponential spacing}
\label{ssec:bw}

In this section we identify the unit circle with the half-open interval $[0,1)$, by the inverse of the map $t \mapsto e^{2i\pi t}$. Note that this differs from the convention of Section~\ref{sec:charac} where the circle was identified to $(0,1]$; this discrepancy is due to convenience of notation in both cases.

Let
$p_1,\ldots,p_n$ be $n$ points in general position in $[0,1)$. We apply a global rotation to the $n$ points so that
$p_1=0$. Now the region with label $1$ is defined to be the region containing $e^{i\epsilon}$
for all $\epsilon>0$ small enough. Recall that $p_i'=p_i+\tfrac{1}{2} \mod 1$ for every $1\leq i \leq
n$. Denote by $\calP$ (resp.~$\calP'$) the set of all $p_i$ (resp.~of all $p_i'$) and note that the interval $[0,\tfrac{1}{2})$ contains exactly $n$ elements of $\calP\cup\calP'$. For every $0\leq i\leq n-1$, we define the variables $X_i$ and $\Gamma_i$ such that the following two conditions are satisfied:
\begin{itemize}
 \item $0=X_0<X_1<\cdots<X_{n-1}<\frac{1}{2}$ is an ordering of the intersection of $[0,\tfrac{1}{2})$ with the set $\calP\cup\calP'$;
 \item for every $0\leq i\leq n-1$, $\Gamma_i=\indset{X_i\in\calP}$.
\end{itemize}

Each $X_i$ is called a \emph{black dot} (resp.~\emph{white dot}) if $\Gamma_i=1$ (resp.~$\Gamma_i=0$). It is clear that from the position of the black and white dots we recover the set $\calP$ up to a global rotation. Furthermore, taking the $p_i$ to be i.i.d. uniform on $[0,1)$ induces the probability distribution on dots described as follows:
\begin{itemize}
 \item $X_0=0$ is a black dot;
 \item $(X_1,\ldots,X_{n-1})$ are the ordering statistics of $n-1$ i.i.d. uniform random variables in $[0,\tfrac{1}{2})$;
 \item $(\Gamma_1,\ldots,\Gamma_{n-1})$ are i.i.d. Bernoulli variables of parameter $\tfrac{1}{2}$, independent of the $X_i$.
\end{itemize}
We also adopt the convention that $X_n=\tfrac{1}{2}$ and $\Gamma_n=0$. The $\Gamma_i$'s are called colors.

For every $1\leq i \leq n$ we define $S_i = X_i-X_{i-1}$ to be the spacing between two consecutive dots. The main tool to understand the joint behaviour of the $S_i$'s is the following lemma (see e.g. \cite[Section~$4.1$]{Pyk65} for a proof), which allows us to get rid of the condition that the sum of the variables $S_i$ should be equal to $\tfrac{1}{2}$.

\begin{lemma}[\cite{Pyk65}]
\label{lemma:exp_var}
Fix $n \geq 1$. If $T_1,\ldots,T_n$ are i.i.d. exponential variables of parameter $1$ and $Y_n := \sum_{i=1}^n T_i$, then $(T_1/2Y_n, \ldots, T_n/2Y_n)$ is independent of $Y_n$ and distributed as $(S_1, \ldots, S_n)$.
\end{lemma}

In the remainder of this section we will frequently use the \emph{exponential spacing model}, defined as follows. Let $T_1,\ldots,T_n$ be $n$ i.i.d. $Exp(1)$ variables and, for every $0\leq k\leq n$, define the dot $Y_k=\sum_{i=1}^k T_i$. Define also $\Gamma_1,\ldots,\Gamma_{n-1}$ to be $n-1$ i.i.d. Bernoulli variables of parameter $\tfrac{1}{2}$ independent of the $T_i$, and set in addition $\Gamma_0=1$ and $\Gamma_n=0$. It will also be useful to extend the definition of the $Y_i$ and $\Gamma_i$ to all $-n \leq i \leq 2n-1$ by setting for every $0\leq i\leq n-1$
\[
(Y_{i\pm n},\Gamma_{i\pm n}):=(Y_i \pm Y_n,1-\Gamma_i).
\]
With this extension, every dot $X_i$ with $0\leq i\leq n-1$ has at least one dot of the opposite color to its left and to its right.

By Lemma~\ref{lemma:exp_var}, up to global scale, the variables $Y_i$
are distributed like the variables $X_i$. Hence, replacing the $X_i$ by the $Y_i$ does not change the probability of each bracelet to occur. Thus, we can use the exponential spacing model to prove results about probabilities of bracelets (Proposition~\ref{prop:nicebracelet} and Theorem~\ref{thm:greenexpectation}). In order to prove Theorem~\ref{thm:greenredlengths} about expected lengths of intervals, we need to overcome the problem that the global scale to go between the $X_i$ and the $Y_i$ is random: it is the total length of the interval $Y_n$. This is done via Lemma~\ref{lem:fundamental}. If $\underline{t}=(t_1,\ldots,t_n)\in\R_+^n$ and $\underline{\gamma}=(\gamma_0,\ldots,\gamma_{n-1})\in\{0,1\}^n$ such that $\gamma_0=1$, we define for every $k\in\{0,1,2\}$  $L_k(\underline{t},\underline{\gamma})$ to be the sum of the lengths of the regions of type~$k$ in a bracelet constructed from $n$ points whose spacings (resp.~colors) are given by $\underline{t}$ (resp.~$\underline{\gamma}$).

\begin{lemma}[Transfer lemma for lengths]
\label{lem:fundamental}
Let $n\geq3$ and let $\underline{T}$ and $\underline{\Gamma}$ be the $n$-tuples of spacings and colors in the exponential spacing model. Set also $Y_n=T_1+\cdots+T_n$. For every $k\in\{0,1,2\}$, we have
\begin{equation}
\label{eq:fundamental}
\E \left[ L_k\left(\frac{1}{2Y_n}\underline{T},\underline{\Gamma}\right) \right] = \frac{\E[L_k(\underline{T},\underline{\Gamma})]}{\E[2Y_n]}.
\end{equation}
In particular,
\begin{equation}
\label{eq:transfer}
\E[L_{k,n}]= \frac{\E[L_k(\underline{T},\underline{\Gamma})]}{2n}.
\end{equation}

\end{lemma}

Note that in Equation~\eqref{eq:transfer}, the expectation on the left-hand side refers to a model on a circle of unit length while the expectation on the right-hand side refers to the exponential spacing model.

\begin{proof}[Proof of Lemma~\ref{lem:fundamental}]
The idea behind the proof is simply that the quantity $\E[L_k(\underline{T},\underline{\Gamma})]$ behaves linearly in $Y_n$. More precisely, observe that
\begin{equation}
\E\left[ L_k(\underline{T},\underline{\Gamma}) | Y_n \right] = 2Y_n \E\left[ L_k\left(\frac{1}{2Y_n}\underline{T},\underline{\Gamma}\right) \bigg| Y_n \right].
\end{equation}
By Lemma~\ref{lemma:exp_var}, $\E[L_k(\tfrac{1}{2Y_n}\underline{T},\underline{\Gamma}) | Y_n] $ is a random variable independent of $Y_n$. Taking the expectation on both sides yields Equation~\eqref{eq:fundamental}. The second statement is a consequence of Lemma~\ref{lemma:exp_var} and the fact that $\E[Y_n]=n$.
\end{proof}

In order to reconstruct the bracelet from the colored dots, we partition the configuration space according to the oriented colored dot configuration realized by the colored dots, which we define below.

\begin{definition}
Let $0<x_1<\cdots<x_{n-1}<x_n$ and let $(\gamma_0,\ldots,\gamma_{n-1})\in\{0,1\}^n$ with $\gamma_0=1$. Extend the $x_i$ and $\gamma_i$ to all $-n \leq i \leq 2n-1$ by setting for every $0\leq i\leq n-1$
\[
(x_{i\pm n},\gamma_{i\pm n}):=(x_i \pm x_n,1-\gamma_i).
\]
For every $0\leq i\leq n-1$, we set $d_i=L$ (resp.~$d_i=R$) if the dot of color $1-\gamma_i$ closest to $x_i$ lies to the left (resp.~right) of $x_i$; in Section~\ref{sec:charac} this was termed as $x_i$ looks to the left (resp.~right). The \emph{oriented colored dot configuration} (OCDC) associated with the $x_i$ and $\gamma_i$ is the $n$-tuple
\[
(o_0,\ldots,o_{n-1})\in\left\{B_L,B_R,W_L,W_R\right\}^n
\]
where $o_i=B_{d_i}$ (resp.~$o_i=W_{d_i}$) if $\gamma_i=1$ (resp.~$\gamma_i=0$).
\end{definition}

It follows from Lemmas~\ref{lem:match},~\ref{lem:green} and~\ref{lem:red} that the OCDC determines the bracelet. However, each bracelet may be realized by several OCDCs.

A useful tool to shorten computations in the remainder of this section is the notion of Erlang random variables (see e.g. \cite[Chapter 1]{cox}).

\begin{definition}
Let $\lambda>0$ be a real number and $k\geq1$ be an integer. The real random variable $U$ is said to follow the \emph{Erlang distribution} of parameters $(k,\lambda)$ if its density with respect to the Lebesgue measure on $\R$ is given by
\[
f(x,k,\lambda)=\frac{\lambda^k x^{k-1} e^{-\lambda x}}{(k-1)!}\indset{x>0}.
\]
We write this as $U\sim$ Erlang$(k,\lambda)$.
\end{definition}

An Erlang$(k,\lambda)$ variable is distributed like the sum of $k$ i.i.d. Exp($\lambda$) variables.

\begin{remark}
\label{rem:computations}
The rest of this section is very computational. We explain here our method to perform the exact computations of probabilities and expectations. We first express them using independent Erlang variables that should verify certain inequalities. This allows us to write them as multiple integrals of the product of the densities of these Erlang variables on the domain defined by the inequalities.  The computation of the multiple integrals is then straightforward  using that a primitive of $f(x,k,\lambda)$ is given for $x>0$ by:
\[
F(x,k,\lambda)=-\sum_{j=0}^{k-1}\frac{(\lambda x)^{j} e^{-\lambda x}}{j!}.
\]
We computed all these integrals by hand. Each computation is elementary, but may take a page or two when there is a sum of two or three terms in the integrand, as is the case for example in Subsection~\ref{ssec:regionlengths}.
\end{remark}

\subsection{Probabilities of individual bracelets}
\label{ssec:individual}

We first prove that the probability of each bracelet is a rational number. This proof does not use the exponential spacing model.

\begin{proof}[Proof of Proposition~\ref{prop:rational}]
We denote by $0=p_1<p_2<\cdots<p_n<1$ the positions of the $n$ points. For every $1\leq i \leq n$ recall the definitions of $l_i=\tfrac{p_i+p_{i+1}}{2}$ and $l'_i=l_i+\tfrac{1}{2}$, where we take the representative modulo $1$ in $[0,1)$. Let $b\in\calB_n$ be a bracelet. Writing that $b$ is achieved is a logical statement that can be written as a disjunction (using the operator ``or'') of disjoint clauses, where each clause corresponds to an ordering of the $3n$ elements
\[
p_1,\ldots,p_n,l_1,\ldots,l_n,l'_1,\ldots,l'_n.
\]
Such an ordering can itself be expressed as a disjunction of disjoint literals, where each literal is a conjunction (using the operator ``and'') of inequalities of the following form: some linear combination with rational coefficients of the $p_i$'s is greater than some rational number. Hence each literal defines a convex polytope contained inside $[0,1]^n$ and whose faces are hyperplanes defined by Cartesian equations involving only rational coefficients. Such a polytope has vertices with rational coordinates, hence its volume is rational. Finally, the probability of $b$ can be written as the sum of the volumes of such convex polytopes, hence is rational.
\end{proof}

We now turn to the computation of the probability of the bracelet $b_n$, which is the equivalence class of $(1,0,1,\ldots,1,0,\ldots,0)$. The computation is much easier for this bracelet than for other bracelets, since it can only be realized by a small number of OCDCs.

\begin{proof}[Proof of Proposition~\ref{prop:nicebracelet}]
The bracelet $b_n$ has exactly two antipodal regions of type~$2$, one with an element of $\calP$ to the left of an element of $\calP'$, the other with an element of $\calP$ to the right of an element of $\calP'$. We rotate the circle so that the element of $\calP$ which is on the left in a region of type~$2$ lies at $0$. Thus, by Lemmas~\ref{lem:match} and~\ref{lem:green}, any OCDC $(o_0,\ldots,o_{n-1})$ realizing $b_n$ must satisfy $o_0=B_R$ and $o_1=W_L$. There are only four words in the equivalence class of $b_n$ such that the first letter corresponds to a region of type~$2$, these are
\begin{align}
\underline{w}_1&=(1,\ldots,1,0,1,0,\ldots,0) \\
\underline{w}_2&=(1,0,1,\ldots,1,0,\ldots,0) \\
\underline{w}_3&=(1,0,\ldots,0,1,\ldots,1,0) \\
\underline{w}_4&=(1,0,\ldots,0,1,0,1,\ldots,1).
\end{align}
Here the number of missing $0$'s and $1$'s represented by the dots is entirely determined by the fact that each word has $n$ letters of each type. By Remark~\ref{rem:signaturearrows}, the sequence $(d_i)_{1\leq i\leq n}$ of $R$'s and $L$'s associated to each of these words is entirely determined by their signatures. For general realizable words, the only remaining ambiguity to entirely determine the OCDC is to know the color of the dot on the left in each region of type~$2$. Here this indeterminacy is lifted by the condition which we imposed, that we have a black dot at position $0$ which is the left dot in a region of type~$2$. The four OCDCs thus obtained from the four $\underline{w}_i$ can be computed explicitly, a case-by-case analysis shows that they all satisfy $o_0=B_R$, $o_1=W_L$ and $o_2=\cdots=o_{n-1}$. We denote these four OCDCs by $\Omega_{B_L}$, $\Omega_{B_R}$, $\Omega_{W_L}$ and $\Omega_{W_R}$, where the index of $\Omega$ corresponds to the value of $o_2$. They correspond respectively to $\underline{w}_1,\underline{w}_2,\underline{w}_3$ and $\underline{w}_4$.

We compute the probabilities of these OCDCs using the exponential spacing model with spacings $T_1,\ldots,T_n$ and colors $\Gamma_0,\ldots,\Gamma_{n-1}$. There is a probability $1/2^{n-1}$ for the variables $\Gamma_1,\ldots,\Gamma_{n-1}$ to achieve the colors imposed by a given OCDC. We recall that in the exponential spacing model $\Gamma_0$ is always fixed to be black.

We first explain the computation of $\P(\Omega_{B_L})$. The condition $o_1=W_L$ translates to $T_1<T_2$. The condition $o_{n-1}=B_L$ translates to $T_2+\cdots+T_{n-1}<T_n$. The condition $o_0=B_R$ translates to $T_1<T_2+\cdots+T_n$ which is implied by the earlier condition that $T_1<T_2$. Finally each condition $o_i=B_L$ for $2\leq i\leq n-2$ translates to $T_2+\cdots+T_i<T_n$, which is implied by the earlier condition that $T_2+\cdots+T_{n-1}<T_n$. Hence the two conditions $T_1<T_2$ and $T_2+\cdots+T_{n-1}<T_n$ together with the conditions on the colors $\Gamma_1,\ldots,\Gamma_{n-1}$ are equivalent to the realization of $\Omega_{B_L}$. Since the $T_i$'s are independent of the $\Gamma_i$'s, we obtain the following product of probabilities:
\[
\P(\Omega_{B_L})=\frac{1}{2^{n-1}}\P(T_1<T_2 \text{ and } T_2+\cdots+T_{n-1}<T_n).
\]

A similar case-by-case reasoning yields
\begin{align}
\P(\Omega_{B_R})&=\frac{1}{2^{n-1}}\P(T_1<T_n \text{ and } T_3+\cdots+T_n<T_2) \\
\P(\Omega_{W_L})&=\frac{1}{2^{n-1}}\P(T_2+\cdots+T_{n-1}<T_n) \\
\P(\Omega_{W_R})&=\frac{1}{2^{n-1}}\P(T_3+\cdots+T_n<T_2).
\end{align}
Since the $T_i$ are i.i.d. we have that $\P(\Omega_{B_L})=\P(\Omega_{B_R})$ and $\P(\Omega_{W_L})=\P(\Omega_{W_R})$.

Set $U=T_3+\cdots+T_{n-1}$, then $U\sim$ Erlang$(n-3,1)$. Write also $X=T_1$, $Y=T_2$ and $Z=T_n$. By the independence of $X,Y,Z$ and $U$ and using the method outlined in Remark~\ref{rem:computations}, we have
\begin{align}
2^{n-1}\P(\Omega_{B_L})&= \P(X<Y \text{ and } Y+U<Z) \\
&=\int_{u=0}^{\infty}\frac{u^{n-4}}{(n-4)!}e^{-u}\int_{y=0}^{\infty}  e^{-y}\int_{z=y+u}^{\infty} e^{-z}\int_{x=0}^y  e^{-x}\d x \, \d z \, \d y \, \d u \\
&= \frac{1}{3\cdot 2^{n-2}},
\end{align}
thus $\P(\Omega_{B_L})=\frac{1}{3\cdot 2^{2n-3}}$.

Set $V=T_2+\cdots+T_{n-1}$, then $V\sim$ Erlang$(n-2,1)$. Write also $Z=T_n$. Then we have
\begin{align}
2^{n-1}\P(\Omega_{W_L})&= \P(Z>V) \\
&=\int_{z=0}^{\infty} e^{-z} \int_{v=0}^z \frac{v^{n-3}}{(n-3)!}e^{-v}\d v \, \d z \\
&= \frac{1}{2^{n-2}},
\end{align}
thus $\P(\Omega_{W_L})=\frac{1}{2^{2n-3}}$.

Since there are $n$ possible choices for the point to place in $0$, we have
\[
\P(b_n)=n\left(\P(\Omega_{B_L})+\P(\Omega_{B_R})+\P(\Omega_{W_L})+\P(\Omega_{W_R})\right),
\]
which yields the desired quantity.
\end{proof}

\subsection{Expected number of regions of each type}
\label{ssec:greenexpectation}

In this subsection we prove Theorem~\ref{thm:greenexpectation} about the expected number of regions of each type.

\begin{proof}[Proof of Theorem~\ref{thm:greenexpectation}]
As described in the introduction, it suffices to study $H_{2,n}$. Fix a point $p\in\calP$. Denote by $f_n$ the probability that $p$ lies in a region of type~$2$ and lies to the left of the element of $\calP'$ which is also in this region. Then we have $\E[H_{2,n}]=2nf_n$. To compute $f_n$, we use the exponential spacing model with the random variables $T_i$, $Y_i$ and $\Gamma_i$ constructed as above and denote by $\underline{O}=(O_0,\ldots,O_{n-1})$ the random OCDC obtained from these random variables. Denote by $Q$ the event that $O_0=B_R$ and $O_1=W_L$. Then by Lemmas~\ref{lem:match} and~\ref{lem:green} we have $\P(Q)=f_n$.

For any OCDC $\underline{o}=(o_0,o_1,\ldots,o_{n-1})$, we define
\[
\alpha(\underline{o})=\left\{ 1\leq i\leq n-1 | o_i\in \{B_L,B_R \} \right\}.
\]
If $\alpha(\underline{o})\neq\varnothing$ set $\alpha^-(\underline{o})=\min(\alpha(\underline{o}))$ and $\alpha^+(\underline{o})=\max(\alpha(\underline{o}))$.
Define the events
\[
A_\varnothing=\{\alpha(\underline{O})=\varnothing\}
\]
and for every $k,l\geq1$ such that $k+l\leq n-1$,
\[
A_{k,l}=\{O_0=B_R, O_1=W_L,\alpha^-(\underline{O})=k+1 \text{ and }\alpha^+(\underline{O})=n-l\}.
\]
These events are illustrated in Figure~\ref{fig:length2}.

Observing that $\alpha(\underline{O})=\varnothing$ automatically implies that $O_0=B_R$ and $O_1=W_L$, we have the following partition for $Q$:
\begin{equation}
\label{eq:qpartition}
Q=A_\varnothing \sqcup \bigsqcup_{\substack{k,l\geq1, \\
        k+l\leq n-1}} A_{k,l}.
\end{equation}

Clearly $\P(A_\varnothing)=2^{-(n-1)}$. If $k,l\geq1$ such that $k+l\leq n-1$, then
\[
\P(A_{k,l})=\frac{1}{2^{\min(1+k+l,n-1)}} \ p_{k,l},
\]
where
\[
p_{k,l}=\P\left( T_1 < T_2+\dots+T_{k+1} \text{ and } T_1 <
      T_{n+1-l}+\dots+T_{n} \right).
\]

Set $X=T_1$, $U=T_2+\cdots+T_{k+1}$ and $V=T_{n+1-l}+\cdots+T_{n}$. Then $U\sim$ Erlang$(k,1)$ and $V\sim$ Erlang$(l,1)$ and $X,U$ and $V$ are independent, thus

  \begin{equation}
    \begin{split}
      p_{k,l} &= \P(X<U \text{ and } X<V) \\
      &= \int_{x=0}^{\infty} e^{-x} \int_{u=x}^{\infty} \frac{u^{k-1}}{(k-1)!}e^{-u} \int_{v=x}^{\infty} \frac{v^{l-1}}{(l-1)!}e^{-v}\d v \, \d u \, \d x\\
      & =  \sum_{i=0}^{k-1} \sum _{j=0}^{l-1} \binom{i+j}{j} \frac{1}{3^{i+j+1}}.
    \end{split}
  \end{equation}

For every $0<x<\tfrac{1}{2}$ we define
   \begin{equation}
   \label{eq:defphi}
          \varphi(x) := \sum_{\substack{k,l\geq1 \\
        k+l\leq n-1}}\sum _{\substack{0
        \leq i \leq k-1 \\
0\leq j \leq l-1}}\frac{1}{2^{\min(1+k+l,n-1)}}  \binom{i+j}{j}x^{i+j+1}.
  \end{equation}

Then we have
\[
f_n=\frac{1}{2^{n-1}}+\varphi\left(\frac13\right)
\]

Lemma~\ref{lem:phix} (stated and proved below) implies that
\begin{equation}
\label{eq:phi13}
\varphi(\tfrac13)=\tfrac14-\tfrac1{2^{n-1}}+\tfrac1{4\cdot 3^{n-2}}
\end{equation}
thus $f_n=\tfrac14+\tfrac1{4\cdot 3^{n-2}}$ and finally
\[
\E[H_{2,n}]=\frac{n}{2}\left( 1+\frac{1}{3^{n-2}} \right) .
\]
\end{proof}

\begin{lemma}
\label{lem:phix}
For every $0<x<\tfrac{1}{2}$ we have
  \begin{equation}
    \varphi(x)= \frac{x}{2} \frac{1-x^{n-2}}{1-x} -
          \frac{x}{2^{n-1}} \frac{1-(2x)^{n-2}}{1-2x}.
  \end{equation}
\end{lemma}

\begin{proof}
  By summing first on $(i,j)$ in formula~\eqref{eq:defphi}, we get
  \begin{equation}
    \varphi(x) =\sum_{\substack{i,j\geq0 \\ i+j \leq n-3}} \binom{i+j}{j}  x^{i+j+1} \left(
      \sum_{\substack{i+1\leq k, \ j+1\leq l, \\
          k+l\leq n-1}} \\\frac{1}{2^{\min(1+k+l,n-1)}} \right).
  \end{equation}
  The inner sum can be computed by changing one variable to
  $s=k+l$. For any $s$ satisfying $i+j+2\leq s \leq n-1$, there are $s-(i+j+1)$ terms
  $(k,l)$ contributing, hence the inner sum is (treating the case
  $s=n-1$ apart):
  \begin{equation}
    \frac{n-1-(i+j+1)}{2^{n-1}} + \sum_{s=i+j+2}^{n-2}\frac{s-(i+j+1)}{2^{1+s}} .
  \end{equation}
  This computation is standard, and the result is $\tfrac{1}{2^{i+j+1}}
  - \tfrac{1}{2^{n-1}}$. Thus
  \begin{equation}
 \varphi(x) =\sum_{\substack{i,j\geq0 \\ i+j \leq n-3}} \binom{i+j}{j}  \left(\frac{x}{2}\right)^{i+j+1}-\frac{1}{2^{n-1}}\sum_{\substack{i,j\geq0 \\ i+j \leq n-3}} \binom{i+j}{j} x^{i+j+1}.
  \end{equation}
  Summing instead on the variables $u=i+j$ and $j$, this becomes
\[
      \varphi(x) =  \sum_{u=0}^{n-3} \left(\frac x2 \right)^{u+1}
      \sum_{j=0}^u \binom{u}{j} \ - \
      \frac{1}{2^{n-1}}  \sum_{u=0}^{n-3} x^{u+1} \sum_{j=0}^u
      \binom{u}{j},
\]
from which the desired formula quickly follows.

\end{proof}

\subsection{Expected total length of regions of type~$k$}
\label{ssec:regionlengths}

Since $\E[L_{0,n}]+\E[L_{1,n}]+\E[L_{2,n}]=1$, to prove Theorem~\ref{thm:greenredlengths} it suffices to compute $\E[L_{2,n}]$ and $\E[L_{1,n}]$, which is done respectively in Subsections~\ref{sssec:type2} and~\ref{sssec:type1}.

\subsubsection{Expected total length of regions of type~$2$}
\label{sssec:type2}

\begin{proposition}
For every $n\geq3$, we have
\[
\E[L_{2,n}]=\frac{3^n+2n+11}{8\cdot 3^{n-1}}.
\]
\end{proposition}

\begin{proof}
We use the exponential spacing model again. Regions of type~$2$ come
in antipodal pairs of equal lengths and exactly one of the
regions in each pair has the element of $\calP$ to the left of the
element of $\calP'$. Recall from the proof of
Theorem~\ref{thm:greenexpectation} that $Q$ denotes the event that
$O_0=B_R$ and $O_1=W_L$. Conditionally on $Q$ being realized, the
region of type~$2$ containing the origin can be divided into three:
the portion from the left boundary to the black dot, the portion from
the black dot to the white dot and the portion from the white dot to
the right boundary. By symmetry, conditionally on $Q$, the expected
lengths of the first and third portions are equal. The position of the
white dot is $T_1$ and we denote by $\rho$ the position of the right
boundary. See Figure~\ref{fig:length2}.

\begin{figure}[h]
  \centering
  \def\svgwidth{11cm}

  \begingroup%
  \makeatletter%
  \providecommand\color[2][]{%
    \errmessage{(Inkscape) Color is used for the text in Inkscape, but the package 'color.sty' is not loaded}%
    \renewcommand\color[2][]{}%
  }%
  \providecommand\transparent[1]{%
    \errmessage{(Inkscape) Transparency is used (non-zero) for the text in Inkscape, but the package 'transparent.sty' is not loaded}%
    \renewcommand\transparent[1]{}%
  }%
  \providecommand\rotatebox[2]{#2}%
  \newcommand*\fsize{\dimexpr\f@size pt\relax}%
  \newcommand*\lineheight[1]{\fontsize{\fsize}{#1\fsize}\selectfont}%
  \ifx\svgwidth\undefined%
    \setlength{\unitlength}{183.85544424bp}%
    \ifx\svgscale\undefined%
      \relax%
    \else%
      \setlength{\unitlength}{\unitlength * \real{\svgscale}}%
    \fi%
  \else%
    \setlength{\unitlength}{\svgwidth}%
  \fi%
  \global\let\svgwidth\undefined%
  \global\let\svgscale\undefined%
  \makeatother%
  \begin{picture}(1,0.42771129)%
    \lineheight{1}%
    \setlength\tabcolsep{0pt}%
    \put(0.42892454,0.26915143){\color[rgb]{0,0,0}\makebox(0,0)[t]{\lineheight{1.25}\smash{\begin{tabular}[t]{c}$0$\end{tabular}}}}%
    \put(0.55141915,0.26885394){\color[rgb]{0,0,0}\makebox(0,0)[t]{\lineheight{1.25}\smash{\begin{tabular}[t]{c}$T_1$\end{tabular}}}}%
    \put(0.61227282,0.30948951){\color[rgb]{0,0,0}\makebox(0,0)[t]{\lineheight{1.25}\smash{\begin{tabular}[t]{c}$\rho$\end{tabular}}}}%
    \put(0.67379791,0.26885394){\color[rgb]{0,0,0}\makebox(0,0)[t]{\lineheight{1.25}\smash{\begin{tabular}[t]{c}$T_1+T_2$\end{tabular}}}}%
    \put(0.26586864,0.26885394){\color[rgb]{0,0,0}\makebox(0,0)[t]{\lineheight{1.25}\smash{\begin{tabular}[t]{c}$-T_n$\end{tabular}}}}%
    \put(0.12318337,0.37936757){\color[rgb]{0,0,0}\makebox(0,0)[t]{\lineheight{1.25}\smash{\begin{tabular}[t]{c}$\dots$\end{tabular}}}}%
    \put(0.83031482,0.37861345){\color[rgb]{0,0,0}\makebox(0,0)[t]{\lineheight{1.25}\smash{\begin{tabular}[t]{c}$\dots$\end{tabular}}}}%
    \put(0.42892454,0.00399738){\color[rgb]{0,0,0}\makebox(0,0)[t]{\lineheight{1.25}\smash{\begin{tabular}[t]{c}$0$\end{tabular}}}}%
    \put(0.55141917,0.0036998){\color[rgb]{0,0,0}\makebox(0,0)[t]{\lineheight{1.25}\smash{\begin{tabular}[t]{c}$T_1$\end{tabular}}}}%
    \put(0.60411423,0.04433512){\color[rgb]{0,0,0}\makebox(0,0)[t]{\lineheight{1.25}\smash{\begin{tabular}[t]{c}$\rho$\end{tabular}}}}%
    \put(0.66563932,0.0036998){\color[rgb]{0,0,0}\makebox(0,0)[t]{\lineheight{1.25}\smash{\begin{tabular}[t]{c}$T_1+T_2$\end{tabular}}}}%
    \put(0.30635829,0.00449216){\color[rgb]{0,0,0}\makebox(0,0)[t]{\lineheight{1.25}\smash{\begin{tabular}[t]{c}$-T_n$\end{tabular}}}}%
    \put(0.25715767,0.11348489){\color[rgb]{0,0,0}\makebox(0,0)[t]{\lineheight{1.25}\smash{\begin{tabular}[t]{c}$\dots$\end{tabular}}}}%
    \put(0.71811345,0.11369418){\color[rgb]{0,0,0}\makebox(0,0)[t]{\lineheight{1.25}\smash{\begin{tabular}[t]{c}$\dots$\end{tabular}}}}%
    \put(0.90595172,0.00363066){\color[rgb]{0,0,0}\makebox(0,0)[t]{\lineheight{1.25}\smash{\begin{tabular}[t]{c}$T_1+\dots+T_{k+1}$\end{tabular}}}}%
    \put(0.10233119,0.00363064){\color[rgb]{0,0,0}\makebox(0,0)[t]{\lineheight{1.25}\smash{\begin{tabular}[t]{c}$-(T_n+\dots+T_{n+1-l})$\end{tabular}}}}%
    \put(0,0.02){\includegraphics[width=\unitlength,page=1]{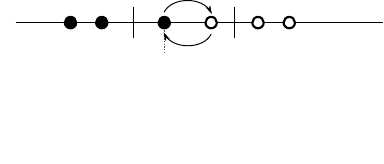}}%
    \put(0,0.02){\includegraphics[width=\unitlength,page=2]{length2.pdf}}%
    \put(0,0.02){\includegraphics[width=\unitlength,page=3]{length2.pdf}}%
    \put(0,0.02){\includegraphics[width=\unitlength,page=4]{length2.pdf}}%
    \put(0,0.02){\includegraphics[width=\unitlength,page=5]{length2.pdf}}%
  \end{picture}%
\endgroup%

  \caption{Arrangement and positions of the relevant points under $A_{\varnothing}$
    (top) and $A_{k,l}$ (bottom).}
  \label{fig:length2}
\end{figure}

By Equation~\eqref{eq:transfer}, we have
\[
\E[L_{2,n}]= \frac{\E[L_2(\underline{T},\underline{\Gamma})]}{2n},
\]
where the expectation on the left-hand side refers to a model on a circle of unit length while the expectation on the right-hand side refers to the exponential spacing model. On the other hand, since in the exponential spacing model there are $n$ choices for the point to place in $0$, we have
\[
\E[L_2(\underline{T},\underline{\Gamma})]=2n\E[(2\rho-T_1)\ind{Q}],
\]
thus
\[
\E[L_{2,n}]=\E[(2\rho-T_1)\ind{Q}].
\]

We use again the partition of $Q$ given by Equation~\eqref{eq:qpartition}. Conditionally on $A_\varnothing$, we have
\[
\rho=\frac{2T_1+T_2}{2}
\]
so
\[
\E[(2\rho-T_1)\ind{A_\varnothing}]=\frac{1}{2^{n-1}}\E[T_1+T_2]=\frac{1}{2^{n-2}}.
\]

Let $k,l\geq1$ be such that $k+l\leq n-1$ and assume for now that $k\geq2$. Conditionally on $A_{k,l}$, we have
\[
2\rho=\min(T_1+\cdots+T_{k+1},2T_1+T_2).
\]
Setting $X=T_1$, $Y=T_2$, $U=T_3+\cdots+T_{k+1}$ and $V=T_{n+1-l}+\cdots+T_{n}$, we have
\[
\E[(2\rho-T_1)\ind{A_{k,l}}]=\frac{1}{2^{\min(1+k+l,n-1)}} \times
\E[\min(Y+U,X+Y)\indset{X<Y+U \text{ and } X<V}].
\]
Since $U\sim$ Erlang$(k-1,1)$, $V\sim$ Erlang$(l,1)$ and $X$, $Y$, $U$ and $V$ are independent, we deduce
\[
2^{\min(1+k+l,n-1)}\E[(2\rho-T_1)\ind{A_{k,l}}]=I_1+I_2,
\]
where
\[
I_1:=\int_{x=0}^{\infty} e^{-x} \int_{u=x}^{\infty} \frac{u^{k-2}}{(k-2)!}e^{-u} \int_{y=0}^{\infty}  (x+y)e^{-y} \int_{v=x}^{\infty} \frac{v^{l-1}}{(l-1)!}e^{-v}\d v \, \d y \, \d u \, \d x
\]
corresponds to the case $X<U$ and
\[
I_2:=\int_{x=0}^{\infty}  e^{-x} \int_{u=0}^{x} \frac{u^{k-2}}{(k-2)!}e^{-u} \int_{y=x-u}^{\infty}  (u+y) e^{-y}\int_{v=x}^{\infty} \frac{v^{l-1}}{(l-1)!}e^{-v}\d v \, \d y \, \d u \, \d x
\]
corresponds to the case $U<X$. A straightforward computation yields
\[
I_1=\sum_{i=0}^{k-2}\sum_{j=0}^{l-1} \binom{i+j}{j} \frac{1}{3^{i+j+1}}+\frac{1}{9}\binom{i+j}{j} \frac{i+j+1}{3^{i+j}}
\]
and
\[
I_2=\sum_{j=0}^{l-1} \binom{j+k-1}{j} \frac{1}{3^{j+k}}+\frac{1}{9}\binom{j+k-1}{j} \frac{j+k}{3^{j+k-1}}
\]
so that
\begin{multline}
\label{eq:k2}
2^{\min(1+k+l,n-1)}\E[(2\rho-T_1)\ind{A_{k,l}}]= \\
\sum_{i=0}^{k-1}\sum_{j=0}^{l-1} \binom{i+j}{j} \frac{1}{3^{i+j+1}}+\frac{1}{9}\binom{i+j}{j} \frac{i+j+1}{3^{i+j}}.
\end{multline}
If $k=1$, then conditionally on $A_{1,l}$ we have $2\rho=T_1+T_2$ so that
\[
\E[(2\rho-T_1)\ind{A_{1,l}}]=\frac{1}{2^{\min(2+l,n-1)}}
\E[T_2\indset{T_1<T_2 \text{ and } T_1<T_{n+1-l}+\cdots+T_n}].
\]
Similar computations as above show that formula~\eqref{eq:k2} is also valid for $k=1$.

Recalling the definition of $\varphi(x)$ from Equation~\eqref{eq:defphi}, we observe that
\[
\E[(2\rho-T_1)\ind{Q}]=\frac{1}{2^{n-2}}+\varphi\left(\frac13\right)+\frac19\varphi'\left(\frac13\right).
\]
Applying formula~\eqref{eq:phi13} to express $\varphi(\tfrac13)$ and using Lemma~\ref{lem:phix} to compute
\[
\varphi'\left(\frac13\right)=\frac98-\frac{9}{2^{n-1}}+\frac{2n+5}{8\cdot 3^{n-3}},
\]
we conclude that
\[
\E[L_{2,n}]=\frac{3^n+2n+11}{8\cdot 3^{n-1}}.
\]
\end{proof}

\subsubsection{Expected total length of regions of type~$1$}
\label{sssec:type1}

\begin{proposition}
For every $n\geq3$, we have
\[
  \E\left[ L_{1,n} \right]=\frac{3^{n-1}-n-1}{2\cdot 3^{n-1}}.
\]
\end{proposition}

\begin{proof}
Denote by $S$ the event that $O_0=B_L$ and $O_{n-1}\neq B_R$. By Lemma~\ref{lem:green}, this is equivalent to requiring that the dot at position $0$ is a black dot looking to the left and lying in a region of type~$1$. Denote by $\lambda$ (resp.~$\rho$) the position of the left (resp.~right) boundary of that region. Since regions of type~$1$ come in antipodal pairs of equal length with exactly one which is occupied, and the point is equally likely to look to the left or to the right, Lemma~\ref{lem:fundamental} and a reasoning similar to the one for $L_{2,n}$ yields
\[
\E[L_{1,n}]=\E[(2\rho-2\lambda)\ind{S}],
\]
where the expectation on the left-hand side refers to a model on a circle of unit length while the expectation on the right-hand side refers to the exponential spacing model.

Recall the definitions of $\alpha^+$ and $\alpha^-$ from the proof of Theorem~\ref{thm:greenexpectation} and define $\beta^+$ and $\beta^-$ similarly for white dots. More precisely, for any OCDC $\underline{o}=(o_0,o_1,\ldots,o_{n-1})$, we define
\[
\beta(\underline{o})=\left\{ 1\leq i\leq n-1 | o_i\in \{W_L,W_R \} \right\}.
\]
If $\beta(\underline{o})\neq\varnothing$, set $\beta^-(\underline{o})=\min(\beta(\underline{o}))$ and $\beta^+(\underline{o})=\max(\beta(\underline{o}))$.
Define the following events.
For every $1\leq l \leq n-2$,
\[
S^1_l:=
\{O_0=B_L, O_{n-1}=B_L, \beta^-(\underline{O})=1, \beta^+(\underline{O})=n-1-l\}.
\]
For every $1\leq l\leq n-3$,
\[
S^2_l:=
\{O_0=B_L,\alpha^+(\underline{O})=n-1-l, \beta^-(\underline{O})=1, \beta^+(\underline{O})=n-1\}.
\]
For every $k,l\geq1$ such that $k+l\leq n-2$,
\[
S^3_{k,l}:=
\{O_0=B_L, O_{n-1}=B_L, \alpha^-(\underline{O})=1, \beta^-(\underline{O})=k+1, \beta^+(\underline{O})=n-1-l\}.
\]
For every $k,l\geq1$ such that $k+l\leq n-3$,
\[
S^4_{k,l}:=
\{O_0=B_L, \alpha^-(\underline{O})=1, \alpha^+(\underline{O})=n-1-l, \beta^-(\underline{O})=k+1, \beta^+(\underline{O})=n-1\}.
\]
For every $1\leq l'\leq n-2$,
\[
S^5_{l'}:=
\{O_0=B_L, \alpha(\underline{O})=\{1,\ldots,n-1-l'\}, \beta(\underline{O})=\{n-l',\ldots,n-1\}\}.
\]
Observing that $O_0=B_L$ implies that $\alpha(\underline{O})\neq\varnothing$ and  that $O_{n-1}=B_L$ implies that $\beta(\underline{O})\neq\varnothing$, we deduce that all the events defined above form a partition of $S$.

We first treat the cases of $S^m$ with $1\leq m \leq 4$. We write $X=T_1$, $U=T_2+\cdots+T_{k+1}$, $V=T_{n-l}+\cdots+T_{n-1}$ and $Y=T_n$. Then $U\sim$ Erlang$(k,1)$, $V\sim$ Erlang$(l,1)$ and $X$, $Y$, $U$ and $V$ are independent.

The event $S^1_l$ corresponds to imposing the colors of $\min(l+2,n-1)$ dots and requiring that $V<Y<X$. Conditionally on $S^1_l$ we have $2\lambda=-V-Y$ and $2\rho=X-Y$. Hence
\[
2^{\min(l+2,n-1)}\E[(2\rho-2\lambda)\ind{S^1_l}]=\int_{v=0}^\infty   \frac{v^{l-1}}{(l-1)!}e^{-v}\int_{y=v}^{\infty}  e^{-y}\int_{x=y}^{\infty}  (x+v)e^{-x} \d x \, \d y \, \d v .
\]
Thus
\[
2^{\min(l+2,n-1)}\E[(2\rho-2\lambda)\ind{S^1_l}]=\frac{9+4l}{4\cdot 3^{l+1}}.
\]

The event $S^2_l$ corresponds to imposing the colors of $\min(l+2,n-1)$ dots and requiring that $V+Y<X$. Conditionally on $S^2_l$ we have $2\lambda=-Y$ and $2\rho=X-Y-V$. Hence
\[
2^{\min(l+2,n-1)}\E[(2\rho-2\lambda)\ind{S^1_l}]=\int_{v=0}^\infty   \frac{v^{l-1}}{(l-1)!}e^{-v} \int_{y=0}^{\infty}  e^{-y} \int_{x=v+y}^{\infty}  (x-v)e^{-x} \d x \, \d y \, \d v.
\]
Thus
\[
2^{\min(l+2,n-1)}\E[(2\rho-2\lambda)\ind{S^2_l}]=\frac{3}{2^{l+2}}.
\]

The event $S^3_{k,l}$ corresponds to imposing the colors of $\min(k+l+2,n-1)$ dots and requiring that $V<Y<X+U$. Conditionally on $S^3_l$ we have $2\lambda=-V-Y$ and $2\rho=\min(X,X+U-Y)$. Hence
\[
2^{\min(k+l+2,n-1)}\E[(2\rho-2\lambda)\ind{S^3_{k,l}}]=I_1+I_2
\]
where
\[
I_1=\int_{x=0}^\infty  e^{-x}  \int_{y=0}^{\infty}  e^{-y}\int_{u=y}^{\infty}  \frac{u^{k-1}}{(k-1)!}e^{-u} \int_{v=0}^y   \frac{v^{l-1}}{(l-1)!}(x+v+y)e^{-v}\d v \, \d u \, \d y \, \d x
\]
corresponds to the case $Y<U$ and
\[
I_2=\int_{y=0}^\infty  e^{-y}  \int_{u=0}^{y}   \frac{u^{k-1}}{(k-1)!}e^{-u} \int_{v=0}^{y}   \frac{v^{l-1}}{(l-1)!}e^{-v}\int_{x=y-u}^\infty   (x+v+u)e^{-x}\d x \, \d v \, \d u \, \d y
\]
corresponds to the case $Y>U$.
Computations yield
\[
I_1=l+2-\frac{k+2l+4}{2^{k+1}}-\sum_{i=0}^{k-1}\binom{i+l}{l}\frac{l}{3^{i+l+1}}-\sum_{i=0}^{k-1}\sum_{j=0}^{l-1}\binom{i+j}{j}\frac{3l+i+j+4}{3^{i+j+2}}
\]
and
\[
I_2=\frac{k+2l+3}{2^{k+2}}-\binom{k+l}{l}\frac{l}{3^{k+l+1}}-\sum_{j=0}^{l-1}\binom{k+j}{j}\frac{3l+k+j+4}{3^{k+j+2}}
\]

\begin{multline}
2^{\min(k+l+2,n-1)}\E[(2\rho-2\lambda)\ind{S^3_{k,l}}]=\\
l+2-\frac{k+2l+5}{2^{k+2}} -\sum_{i=0}^{k}\binom{i+l}{l}\frac{l}{3^{i+l+1}}-\sum_{i=0}^{k}\sum_{j=0}^{l-1}\binom{i+j}{j}\frac{3l+i+j+4}{3^{i+j+2}}.
\end{multline}

The event $S^4_{k,l}$ corresponds to imposing the colors of $\min(k+l+2,n-1)$ dots and requiring that $V+Y<X+U$. Conditionally on $S^4_l$ we have $2\lambda=-Y$ and $2\rho=\min(X,X+U-Y-V)$. Hence
\[
2^{\min(k+l+2,n-1)}\E[(2\rho-2\lambda)\ind{S^4_{k,l}}]=I_3+I_4
\]
where
\[
I_3=\int_{x=0}^\infty  e^{-x}  \int_{y=0}^{\infty}  e^{-y} \int_{v=0}^{\infty}   \frac{v^{l-1}}{(l-1)!}e^{-v} \int_{u=y+v}^\infty   \frac{u^{k-1}}{(k-1)!}(x+y)e^{-u}\d u \, \d v \, \d y \, \d x
\]
corresponds to the case $Y+V<U$ and
\[
I_4=\int_{y=0}^\infty  e^{-y} \int_{v=0}^{\infty}   \frac{v^{l-1}}{(l-1)!}e^{-v} \int_{u=0}^{y+v}  \frac{u^{k-1}}{(k-1)!}e^{-u} \int_{x=y+v-u}^\infty   (x+u-v)e^{-x}\d x \, \d u \, \d v \, \d y 
\]
corresponds to the case $Y+V>U$.
Computations yield
\[
I_3=\sum_{i=0}^{k-1}\sum_{h=0}^i\binom{l+h-1}{h}\frac{3+i-h}{2^{l+i+2}}
\]
and
\[
I_4=\sum_{h=0}^k\binom{l+h-1}{h}\frac{3+k-h}{2^{l+k+2}}
\]
thus
\[
2^{\min(k+l+2,n-1)}\E[(2\rho-2\lambda)\ind{S^4_{k,l}}]=
\sum_{i=0}^{k}\sum_{h=0}^i\binom{l+h-1}{h}\frac{3+i-h}{2^{l+i+2}}.
\]
The cases $S^1_l$ (resp.~$S^2_l$) correspond to $S^3_{0,l}$ (resp.~$S^4_{0,l}$) provided we extend the definitions of $S^3_{k,l}$ and $S^4_{k,l}$ to $k=0$.

For the case of $S^5_{l'}$, we define $X=T_1$, $U=T_2+\cdots+T_{n-1-l'}$, $Z=T_{n-l'}$, $V=T_{n-l'+1}+\cdots+T_{n-1}$ and $Y=T_n$. Then $U\sim$ Erlang$(n-2-l',1)$, $V\sim$ Erlang$(l'-1,1)$ and $X$, $Y$, $U$, $V$ and $Z$ are independent. The event $S^5_{l'}$ corresponds to requiring that $V+Y<X+U$ and imposing the colors of $n-1$ dots. Conditionally on $S^5_{l'}$ we have $2\lambda=-Y$ and $2\rho=\min(X,X+U-Y-V)$. Setting $k=n-2-l'$ and $l=l'-1$ and comparing with the computation for $S^4_{k,l}$ we deduce that if $2\leq l'\leq n-3$:
\begin{equation}
\label{eq:s5}
2^{n-1}\E[(2\rho-2\lambda)\ind{S^5_{l'}}]=
\sum_{i=0}^{n-2-l'}\sum_{h=0}^i\binom{l'+h-2}{h}\frac{3+i-h}{2^{l'+i+1}}.
\end{equation}
A separate computation shows that Equation~\eqref{eq:s5} still holds when $l'=n-2$. So the total contributions of $S^5_{l'}$ for $2\leq l'\leq n-2$ correspond to the contributions of $S^4_{k,l}$ with $k+l=n-3$ and $0\leq k\leq n-4$. For $l'=1$, another separate computation gives
\[
2^{n-1}\E[(2\rho-2\lambda)\ind{S^5_1}]=\frac{2^n-n-2}{2^{n-1}}.
\]

For every $s\geq1$, we denote by $G_s$ the set of all pairs of integers $(k,l)$ such that $k\geq0$, $l\geq1$ and $k+l\leq s$. Then $\E[L_{1,n}]=A_1-A_2-A_3+A_4+A_5$ with
\begin{align}
  A_1 &= \sum_{(k,l)\in
  G_{n-2}}\frac{1}{2^{\min(k+l+2,n-1)}}\left(l+2-\frac{k+2l+5}{2^{k+2}}\right) \\
  A_2 &= \sum_{(k,l)\in
  G_{n-2}}\frac{1}{2^{\min(k+l+2,n-1)}} \sum_{i=0}^k
       \binom{i+l}{l}\frac{l}{3^{i+l+1}} \\
  A_3 &= \sum_{(k,l)\in  G_{n-2}}\frac{1}{2^{\min(k+l+2,n-1)}} \sum_{i=0}^k
  \sum_{j=0}^{l-1} \binom{i+j}{j}\frac{3l+i+j+4}{3^{i+j+2}} \\
  A_4 &= \sum_{(k,l)\in G_{n-3}}\frac{1}{2^{k+l+2}}\sum_{i=0}^k \sum_{h=0}^i \binom{l+h-1}{h}\frac{3+i-h}{2^{l+i+2}} \\
A_5 &= \frac{2^n-n-2}{2^{2n-2}}+\frac{1}{2^{n-1}} \sum_{k=0}^{n-4} \sum_{i=0}^k \sum_{h=0}^i \binom{n-k+h-4}{h}\frac{3+i-h}{2^{n-k+i-1}}.
\end{align}

Computing these sums can be done via lengthy but elementary manipulations
of indices and summation formulas, analogous to those of Lemma~\ref{lem:phix}. Here we simply indicate the auxiliary summation variables needed and the order in which to compute the sums. For $A_4$ it is useful to introduce
$u=l+i, v=l+h, s=k+l$ and compute the sums on $s,l,v,u$ in this order (meaning that the sum on $s$ is the innermost sum); $A_5$
can be deduced easily. For $A_1,A_2,A_3$ it is more convenient to split
the terms for which $k+l=n-2$ from the others. Then $A_1$ is
straightforward; $A_2$ can be calculated by introducing $s=k+l, u=i+l$ and computing the
sums on $l,u,s$ in this order; for $A_3$, one can introduce
$s=k+l,u=i+j$ and compute the sums on $l,s,j,u$ in this order. In this
way, at each step the sums to compute are either of the form
$\sum_{m=a}^b m^{d}x^m$, or of the form
$\sum_{m=0}^n \binom{n}{m} m^{d}$, with $d \in
\{0,1,2\}$. These being standard, we omit the computations here.
Thus we get
\begin{align}
  A_1 &= \frac19 \left( 11 - 9\frac{n+1}{2^{n-1}} + \frac{3n+4}{2^{2n-2}} \right) \\
  A_2 &= \frac{1}{16} \left( 1 - \frac{2n-1}{3^{n-1}} \right) \\
  A_3 &= \frac{1}{16} \left( 15 - \frac{n+1}{2^{n-5}} + \frac{10n+7}{3^{n-1}} \right) \\
  A_4 &= \frac{5}{18} - \frac{2n-5}{2^n} - \frac{13+3n}{9\cdot 2^{2n-2}} \\
  A_5 &= \frac{2n-5}{2^n} + \frac{1}{2^{2n-2}}
\end{align}
and combining these yields
\begin{equation*}
  \E\left[ L_{1,n} \right]=\frac{3^{n-1}-n-1}{2\cdot 3^{n-1}}.
\end{equation*}
\end{proof}

\subsection{Equidistribution of regions of each type}
\label{ssec:equidistributed}

This subsection is devoted to the proof of Theorem~\ref{thm:functionalcv}, which is done by making use of Theorems~\ref{thm:greenexpectation} and~\ref{thm:greenredlengths}.

\begin{proof}[Proof of Theorem~\ref{thm:functionalcv}]

Recall that we identify the circle with the interval $[0,1)$, so that, for $0\leq t \leq 1$, the circular arc from $1$ to $e^{2i\pi t}$ gets identified to the interval $[0,t]$. We can restrict ourselves to $0\leq t\leq \tfrac12$, since for $\tfrac12 <t \leq 1$, we have for every $k\in\{0,1,2\}$
\begin{align*}
h_{k,n}(t)&=h_{k,n}\left(t-\frac12\right)+h_{k,n}\left(\frac12\right) \\
\ell_{k,n}(t)&=\ell_{k,n}\left(t-\frac12\right)+\ell_{k,n}\left(\frac12\right).
\end{align*}

Let us first focus on the proof of Theorem~\ref{thm:functionalcv} (i), as (ii) can be showed in the same way. Firstly, remark that the variables $( \frac{h_{0,n}(t)}{2n}, \frac{h_{1,n}(t)}{2n}, \frac{h_{2,n}(t)}{2n} )_{0 \leq t \leq \tfrac12}$ live in the compact space $[0,1]^{[0,\tfrac12]}$, so the sequence that we consider is tight. Thus, we only have to check the convergence of its finite-dimensional marginals. Fix $m \geq 1$ and $0<a_1<a_2<\cdots<a_m<\tfrac12$ to be $m$ real numbers. We also set $a_0=0$ and $a_{m+1}=\tfrac12$. We first prove that the proportions of regions of each color in $[a_0,a_1]\ldots,[a_m,a_{m+1}]$ are ``almost independent'', and then use Theorem~\ref{thm:greenexpectation} on each of these intervals to get the result. For this, let $N_1, \ldots, N_{m+1}$ be the number of dots in each of these intervals. Recall that black (resp.~white) dots correspond to elements of $\calP$ (resp.~$\calP'$). It is clear that $(N_1, \ldots, N_{m+1})$ is distributed as a multinomial of parameters $(n; 2(a_1-a_0), 2(a_2-a_1),\ldots, 2(a_{m+1}-a_m))$. Thus, for any $1 \leq i \leq m+1$, a Chernoff bound provides
\begin{align*}
\P \left( \left\lvert N_i-2n(a_i-a_{i-1}) \right\rvert\geq n^{3/4} \right) \leq 2 e^{-2\sqrt{n}}.
\end{align*}
This implies that, with probability going to $1$ as $n \rightarrow \infty$:
\begin{equation}
\label{eq:bounds}
\forall 1 \leq i \leq m+1, \left\lvert N_i-2n(a_i-a_{i-1})\right\rvert\leq n^{3/4}.
\end{equation}
Let us assume from now on that this holds. We now need to control the interactions between two different intervals of the form $[a_i,a_{i+1}]$. The key remark is the following: among all the lines arising as the boundary of some region entirely contained inside $[a_i,a_{i+1}]$, at most $4$ are not midpoints of two dots that are both inside $[a_i,a_{i+1}]$. These four lines involve the leftmost and rightmost dots of each color in the interval. Hence, overall, at most $4(m+1)$ regions are created by interactions between two intervals. Consider now the affine map sending $a_i$ to $0$ and $a_{i+1}$ to $\tfrac12$. Conditionally to the number of points in the interval $[a_i,a_{i+1}]$, their images by this map are i.i.d. uniform on $[0,\tfrac12]$. By~\eqref{eq:bounds}, using Theorem~\ref{thm:greenexpectation} and the previous key remark, jointly for all $i$:

\begin{align*}
&\left( \frac{h_{0,n}(a_{i+1}) - h_{0,n}(a_i)}{2n}, \frac{h_{1,n}(a_{i+1}) - h_{1,n}(a_i)}{2n}, \frac{h_{2,n}(a_{i+1}) - h_{2,n}(a_i)}{2n} \right)\\ &\underset{n \rightarrow \infty}{\overset{(d)}{\longrightarrow}} \left( \frac{a_{i+1}-a_i}{4}, \frac{a_{i+1}-a_i}{2}, \frac{a_{i+1}-a_i}{4} \right).
\end{align*}
Thus, the finite-dimensional marginals of the process converge, and one gets (i).

Let us now check that the same method may be applied to prove (ii). The sequence involved is tight for the same reason. To prove the convergence of the finite-dimensional marginals of the process, we again use the fact that only a bounded number of regions arise from the interaction between two different intervals.
The only additional ingredient that we need is the following result, whose proof can be found in \cite{Lev39}, which tells us that all the regions are small enough, so that omitting a bounded number of regions does not change much the sum of the lengths of the regions of a given type.

\begin{lemma}[\cite{Lev39}]
\label{lem:levy}
Let $M_n$ be the maximal distance between two consecutive points for $n$ i.i.d. uniform points on $[0,\tfrac12]$. Then, as $n \rightarrow \infty$, in probability:
\begin{align*}
\frac{n M_n}{\log n} \rightarrow \frac{1}{2}.
\end{align*}
\end{lemma}

Considering the first dot to the left and the first dot to the right of a region of type~$k$ with $k\in\{0,1,2\}$, we find that an upper bound for the length of that region is $(k+1)M_n$. By Lemma~\ref{lem:levy}, with probability going to $1$ as $n \rightarrow \infty$, every region length is less than $\sqrt{n}$, and thus one can use the same key argument as in the proof of (i). By Theorem~\ref{thm:greenredlengths}, we get jointly for all $i$:

\begin{align*}
&\left(  \ell_{0,n}(a_{i+1})- \ell_{0,n}(a_{i}), \ell_{1,n}(a_{i+1})- \ell_{1,n}(a_{i}), \ell_{2,n}(a_{i+1})-\ell_{2,n}(a_{i}) \right) \\ &\underset{n \rightarrow \infty}{\overset{(d)}{\longrightarrow}} \left( \frac{a_{i+1}-a_i}{8}, \frac{a_{i+1}-a_i}{2}, \frac{3(a_{i+1}-a_i)}{8} \right).
\end{align*}
Hence, the finite-dimensional marginals of the process converge, and one gets (ii).
\end{proof}

\section{Uniform realizable words and bracelets}
\label{sec:uniformwords}

The aim of this section is to prove Theorem~\ref{thm:refinedconvergence} stated below, which immediately implies Theorem~\ref{thm:convergence} about the asymptotic shape of a uniformly random realizable word.

Let $\underline{w}^{(n)}$ be a random word taken uniformly in the set of
realizable words of length $2n$. We define the \emph{folded word} obtained from $\underline{w}^{(n)}$ to be the word
$\hat{\underline{w}}^{(n)}$ of length $n$ on the alphabet
$\{00,10,01,11\}$,
whose letter in position $i$ is the concatenation $w^{(n)}_i w^{(n)}_{i+n}$. For
$x \in [0,n]$ and $a \in \{ 00,11,10,01 \}$, denote by $S^{a}_{x}$
the number of letters $a$ in $\hat{\underline{w}}^{(n)}$ between positions
$0$ and $\lfloor x \rfloor$. Then the following holds:
\begin{theorem}
\label{thm:refinedconvergence}
\begin{itemize}
\item[(i)] The following holds in probability:
\begin{align*}
\left( \frac{S^{00}_{n}}{n}, \frac{S^{11}_{n}}{n}, \frac{S^{10}_{n}}{n}, \frac{S^{01}_{n}}{n} \right) \underset{n \rightarrow \infty}{\longrightarrow} \left(\frac{1}{6}, \frac{1}{6}, \frac{1}{3}, \frac{1}{3} \right).
\end{align*}

\item[(ii)]
We have the functional convergence:
\begin{align*}
&\frac{2}{\sqrt{n}}\left( S^{00}_{cn}-\frac{cn}{6}, S^{11}_{cn}-\frac{cn}{6}, S^{10}_{cn}-\frac{cn}{3}, S^{01}_{cn}-\frac{cn}{3}\right)_{0 \leq c \leq 1} \\
&\underset{n \rightarrow \infty}{\overset{(d)}{\longrightarrow}} \left( W^{(2)}_c, W^{(2)}_c, W^{(1)}_c-W^{(2)}_c, -W^{(1)}_c-W^{(2)}_c \right)_{0 \leq c \leq 1}
\end{align*}
where $W^{(1)}, W^{(2)}$ are two independent Brownian motions of respective variances $2/3$ and $2/9$.
\end{itemize}
\end{theorem}

To simplify notations, we will often drop the
dependence in $n$. We prove both parts of the theorem at once. The main idea in the proof
is to rephrase it in terms of a random walk, and then use a
\emph{local limit theorem}. A local limit theorem controls the precise
value of a random walk after a large number of steps. Let us state it
properly (see e.g. \cite[Theorem~$6.1$]{Rva61} for a proof of this
result). Recall that a random variable $Y \in \Z^j$ is called \emph{aperiodic} if there is no strict sublattice of $\Z^j$ containing the set of differences $\{x-y, x,y \in \Z^j, \P(Y=x)>0, \P(Y=y)>0 \}$.

\begin{theorem}[\cite{Rva61}]
\label{thm:multivariatelltart1}
Let $j \geq 1$ and $(\textbf{Y}_i)_{i \geq 1} := \left((Y_i^{(1)}, \ldots, Y_i^{(j)})\right)_{i \geq 1}$ be i.i.d. random variables in $\Z^j$ with finite variance, such that the covariance matrix $\Sigma$ of  $\mathbf{Y}_1$ is positive definite. Assume in addition that $\textbf{Y}_1$ is aperiodic and denote by $\textbf{M}$ the mean of $\textbf{Y}_1$. Finally, define for $n \geq 1$
\begin{align*}
\mathbf{T_n} = \frac{1}{\sqrt{n}} \left(\sum_{i=1}^n \mathbf{Y}_i - n \mathbf{M} \right) \in \R^j.
\end{align*}
Then, as $n \rightarrow \infty$, uniformly for $\mathbf{x} \in \R^j$ such that $\P \left( \mathbf{T_n} = \mathbf{x} \right) > 0$,
\begin{align*}
\P \left( \mathbf{T_n} = \mathbf{x} \right) = \frac{1}{(2\pi n)^{j/2} \sqrt{\det \Sigma}} e^{-\frac{1}{2} ^t \bf{x} \Sigma^{-1} \bf{x}} + o\left( n^{-j/2} \right).
\end{align*}
\end{theorem}

We now define a walk from a realizable word, in an almost bijective
way: given any $a := xy \in \{ 00, 01, 10, 11 \}$, define $f(a) =
x-y$. Then, to a folded realizable word $\hat{\underline{w}} := \hat{w}_1 \cdots \hat{w}_n$, we associate the walk $S$ satisfying $S_0=0$ and, for all $i \geq 1$, $S_i-S_{i-1}= f(\hat{w}_i)$.

Remark that occurrences of $01$ (resp.~$10$) in $\hat{\underline{w}}$
correspond to jumps by $-1$ (resp.~$+1$) in $S$. Jumps by $0$ in $S$ may
correspond to either $00$ or $11$, but as these two letters shall
alternate in a folded realizable word, if one knows whether the first
jump $0$ corresponds to $00$ or $11$, then it is possible to recover
$\hat{\underline{w}}$ from $S$. By symmetry, we assume from now on
that the first jump by $0$ corresponds to $11$, so that the map
$\hat{\underline{w}} \mapsto S$ is a bijection from $\hat{\cW}_n^+$ to
$\Walks^+(n)$, where $\hat{\cW}_n^+$ is the set of folded realizable words whose first $11$ appears before the first $00$, and $\Walks^+(n)$ is the set of walks of length $n$, starting from $0$ and with steps in $\{0,+1,-1\}$, with an even nonzero number of steps $0$.

Now take $n$ to be a positive integer. We want to study a uniform element of the set $\Walks^+(n)$. To this end, we first study the set $\Walks(n)$ of walks of length $n$, starting from $0$ and with jumps in $\{0,+1,-1\}$. We define a walk $(T_i)_{0 \leq i \leq n} := (S_i,K_i)_{0 \leq i \leq n}$ on $\Z^2$ as follows: its first coordinate is a uniform element of $\Walks(n)$, $K_0=0$ and, for any $0 \leq i \leq n-1$, $K_{i+1}-K_i = \mathds{1}_{S_{i+1}-S_i=0}$. In other words, the second coordinate of $T$ enumerates the steps $0$ in the walk $S$. It is clear by definition that $(T_i)_{0 \leq i \leq n}$ is a random walk on $\Z^2$ starting from $(0,0)$, with i.i.d. jumps $Y_1, \ldots, Y_n$ whose distribution is the following:
\begin{align*}
\P \left( Y_1=(1,0) \right) = \P \left( Y_1=(-1,0) \right) = \P \left( Y_1 = (0,1) \right) = \frac{1}{3}.
\end{align*}
In particular, $Y_1$ has respective mean and covariance matrix
$$M = \begin{pmatrix}
0 \\
1/3
\end{pmatrix}
\text{ and }
\Sigma=
\left(\begin{matrix}
2/3 & 0\\
0 & 2/9
\end{matrix}\right)$$

We want to prove the functional convergence of the walk $S$, along with the process $(K_i)_{0 \leq i \leq n}$ counting the number of ``$0$'' jumps in the walk, conditionally on $K_n$ being even and nonzero. Since, clearly, $\P \left( K_n = 0 \right) = o(\P(K_n = 0 \mod 2))$, we only need to condition $K_n$ to be even.

In what follows, we define $(S_u)_{u \in [0,n]}$ (resp.~$(K_u)_{u \in [0,n]}$) as the linear interpolation of $(S_i)_{0 \leq i \leq n}$ (resp.~$(K_i)_{0 \leq i \leq n}$) on the whole interval.

\begin{proposition}
\label{prop:convergence}
The following convergence holds in distribution, in $\mathcal{C}([0,1],\R^2)$:
\begin{align*}
\left( \left(\frac{S_{cn}}{\sqrt{n}}, \frac{K_{cn}-cn/3}{\sqrt{n}} \right)_{0 \leq c \leq 1} \Big| K_n=0 \mod 2 \right) \underset{n \rightarrow \infty}{\overset{(d)}{\longrightarrow}} \left( W_c^{(1)}, W_c^{(2)} \right)_{0 \leq c \leq 1}
\end{align*}
where $W^{(1)}, W^{(2)}$ are independent Brownian motions of respective variances $2/3$ and $2/9$.
\end{proposition}

The whole proof of this proposition is highly inspired from the one of \cite[Lemma~$4.1$]{The18}. Let us start with a result on the corresponding unconditioned random walk:

\begin{equation}
\label{eq:expect0}
\left( \frac{S_{cn}}{\sqrt{n}}, \frac{K_{cn}-cn/3}{\sqrt{n}} \right)_{0 \leq c \leq 1} \underset{n \rightarrow \infty}{\overset{(d)}{\longrightarrow}} \left( W^{(1)}_c, W^{(2)}_c \right)_{0 \leq c \leq 1}
\end{equation}

This result is a consequence of Theorem~\ref{thm:multivariatelltart1}. Indeed, by \cite[Theorem~16.14]{Kal02}, it is enough to check that the one-dimensional convergence holds for $c=1$. One gets this from Theorem~\ref{thm:multivariatelltart1}. Uniformly for $a, b$ in a compact subset of $\R$:
\[\P(S_n= \lfloor a \sqrt{n} \rfloor, K_n=\lfloor n/3+b \sqrt{n}\rfloor)  \quad \mathop{\sim}_{n \rightarrow \infty} \quad \frac{1}{{2 \pi n \sqrt{\det \Sigma}}}  e^{-\frac{1}{2} \left( \frac{2}{3} a^2 + \frac{2}{9} b^2 \right)}.\]
This implies (see e.g.~\cite[Theorem~7.8]{Bil68}) that $({S_{ n }}/{\sqrt{n}}, ({K_{n} -  n/3})/{\sqrt{n}})$ converges in distribution to
$(W^{(1)}_1, W^{(2)}_1)$. The convergence~\eqref{eq:expect0} follows.

We now want a conditioned version of~\eqref{eq:expect0}, taking into
account the fact that $K_n$ has to be even. To this end, take $0<u<1$
and take $F: \mathcal{C}([0,u],\R^2) \rightarrow \R$ a bounded
continuous functional.
Set
$$E_n := \E\left[F\left( \frac{S_{cn}}{\sqrt{n}},
    \frac{K_{cn}-cn/3}{\sqrt{n}} \right)_{0 \leq c \leq u} \;\middle|\
  K_n=0 \mod 2\right].$$
Setting $\varphi_n(i) = \P (K_n=i \mod 2)$ and observing that the (unconditioned) walk until time $nu$ is independent of the walk between $nu$ and $n$, one can write:
\begin{equation}
\label{eq:expect1}
E_n = \E\left[F\left( \frac{S_{cn}}{\sqrt{n}}, \frac{K_{cn}-cn/3}{\sqrt{n}} \right)_{0 \leq c \leq u}\frac{\varphi_{n-\lfloor nu \rfloor}(K_{\lfloor nu \rfloor})}{\varphi_n(0)}\right]
\end{equation}

In order to estimate this quantity, simply remark that $K_n$ is
distributed as a binomial $\Bin_n$ of parameters $(n,1/3)$. Now,
remark by a simple computation that $\P(\Bin_n = 0 \mod 2) + \P(\Bin_n
= 1 \mod 2) = 1$, and $\P(\Bin_n = 0 \mod 2) - \P(\Bin_n = 1 \mod 2) =
3^{-n}$, which implies that $\varphi_n(0)$ and $\varphi_n(1)$ both converge to $1/2$ as
$n \rightarrow \infty$.

Thus,~\eqref{eq:expect1} can be rewritten:

\begin{align}
\label{eq:expect2}
E_n &= \E\left[F\left( \frac{S_{cn}}{\sqrt{n}}, \frac{K_{cn}-cn/3}{\sqrt{n}} \right)_{0 \leq c \leq u}\frac{1/2+o(1)}{1/2+o(1)}\right] \nonumber \\
&= \E\left[F\left( \frac{S_{cn}}{\sqrt{n}}, \frac{K_{cn}-cn/3}{\sqrt{n}} \right)_{0 \leq c \leq u}\right] + o(1)
\end{align}
and we get Proposition~\ref{prop:convergence} on $[0,u]$. In order to extend it to the whole interval $[0,1]$, it now suffices to show that the process is tight on $[0,1]$.

\begin{proof}[Proof of tension on the whole interval]
The convergence~\eqref{eq:expect2} shows notably that, conditionally
to the fact that $K_n=0 \mod 2$, the process
\begin{equation}
  \label{eq:process1}( {S_{cn}}/{\sqrt{n}}, ({ K_{cn} - cn/3})/{\sqrt{n}})_{0 \leq c \leq
    1}
\end{equation}
is tight on $[0,u]$ for every $u \in (0,1)$. To show that it is in
addition tight on $[u,1]$, we only need to check that, for $u \in
(0,1)$, the process
$$( {S_{ n-cn }}/{\sqrt{n}}, ({ K_{n-cn} - n(1-c)/3})/{\sqrt{n}})_{0
  \leq c \leq u}$$
is tight  conditionally on $K_n=0 \mod 2$. For this, we use the invariance of the process by time-reversal: the process $(\widehat{S}_i,\widehat{K}_i)_{0 \leq i \leq n}:=(S_n-S_{n-i},K_n-K_{n-i})_{0 \leq i \leq n}$ has the same distribution as $(S_i,K_i)_{0 \leq i \leq n}$, and this is also true under the condition that $K_n=0 \mod 2$. By definition, we can write
\begin{multline}
\left( \frac{S_{n-cn}}{\sqrt{n}}, \frac{ K_{n-cn} - n(1-c)/3}{\sqrt{n}}\right)_{0 \leq c \leq u}= \\
\left( \frac{\widehat{S}_n-\widehat{S}_{cn}}{\sqrt{n}}, \frac{\widehat{K}_n- n/3}{\sqrt{n}}- \frac{\widehat{K}_{cn}-cn/3}{\sqrt{n}}  \right)_{0 \leq c \leq u}.
\end{multline}

Now, letting $\sigma^2 := 2n/9$ be the variance of
$K_1$, we obtain that, uniformly for $b$ in a compact subset of $\R$,
$$\P(K_n=\lfloor n/3+b \sqrt{n}\rfloor \, \big| \, K_n=0 \mod 2) =
\frac{2}{\sqrt{2 \pi n} \sigma} e^{-\frac{b^2}{2\sigma^2}} + o\left(\frac{1}{\sqrt{n}}\right)$$
as $n \rightarrow \infty$.  This implies that, conditionally on $K_n=0
\mod 2$, $(K_n - n/3)/\sqrt{n}$ converges in distribution. Hence, by
\eqref{eq:expect2}, the initial process~\eqref{eq:process1} is tight on $[u,1]$  conditionally on $K_n = 0 \mod 2$.

Finally, the process is tight on $[0,1]$. Furthermore, the convergence of the finite-dimensional marginals is just a consequence of~\eqref{eq:expect2}. This put together implies Proposition~\ref{prop:convergence}.
\end{proof}

We can now prove the main result of this section, Theorem
\ref{thm:refinedconvergence}, by translating Proposition
\ref{prop:convergence} in terms of folded realizable words. For this,
we make use of the following lemma, which relates the behaviour of a
folded realizable word in $\hat{\cW}^+_n$ to the behaviour of the associated element of $\Walks^+(n)$.

\begin{lemma}[From the walk to the word]
\label{lem:walktoword}
Let $\hat{\underline{w}}$ be a folded realizable word in $\hat{\cW}^+_n$ and $(S_i,K_i)_{0 \leq i \leq n}$ be the associated walk on $\Z^2$. For any $i \geq 0$, denote by $\alpha_i$ (resp.~$\beta_i, \gamma_i, \delta_i$) the number of occurrences of $11$ (resp.~$00,10,01$) in the word $\hat{\underline{w}}$ up to position $i$. Then the following holds.
For any $i \geq 0$, any $a \in \Z$, any $p\geq0$:
$$\left\{
    \begin{array}{ll}
        S_i = a \\
        K_i= p
    \end{array}
\right.
\Longleftrightarrow
\left\{
    \begin{array}{ll}
        \alpha_i= \lfloor \frac{p+1}{2} \rfloor, \beta_i = \lfloor \frac{p}{2} \rfloor\\
    \gamma_i= \frac{i-p+a}{2}, \delta_i = \frac{i-p-a}{2}.
    \end{array}
\right.$$
\end{lemma}

This lemma, whose proof is straightforward, implies Theorem~\ref{thm:refinedconvergence}:

\begin{proof}[Proof of Theorem~\ref{thm:refinedconvergence}]
The proof just boils down to putting together Lemma~\ref{lem:walktoword} and Proposition~\ref{prop:convergence}. Indeed, Lemma~\ref{lem:walktoword} (keeping the same notation as in its statement) allows us to write for all $1 \leq i \leq n$:
\begin{equation}
  \begin{split}
    \alpha_i & = \frac{i}{6} +  \frac{K_i-i/3}{2} + c_1 = \beta_i + c_2 \\
\gamma_i & = \frac{i}{3} + \frac{S_i-(K_i-i/3)}{2} \\ \delta_i &= \frac{i}{3} + \frac{-S_i-(K_i-i/3)}{2},
\end{split}
\end{equation}
where $c_1, c_2$ are bounded in absolute value by $1$, independently of $n$ and $i$. This proves point \textit{(ii)} of the theorem by the convergence of Proposition~\ref{prop:convergence}. Using the fact that
$\sup_{0 \leq c \leq 1} |W_c^{(1)}|$, $\sup_{0 \leq c \leq 1} |W_c^{(2)}|$ are bounded in probability, point \textit{(i)} follows.
\end{proof}

\begin{remark}
\label{rem:braceletnecklace}
Notice that the conclusions of Theorems~\ref{thm:convergence} and~\ref{thm:refinedconvergence} still
hold when one considers a uniform realizable bracelet instead of a
uniform realizable word. Indeed, one just has to check that, with
probability going to $1$ as $n \rightarrow \infty$, the equivalence class up to shift and reversal of a uniform
realizable word of size $2n$ has cardinality $2n$. To see this, first recall that we only need to deal with the first $n$ letters of a realizable word of size $2n$, then
remark that a word of length $n$ equal to one of its cyclic shifts is necessarily
periodic, of period at most $n/2$. Thus, there are at most $3^{n/2}$
realizable words with fixed period. Summing over all possible periods,
there are at most $n 3^{n/2}$ such words, which is $o(\# \cW_n)$. Furthermore, a word of length $n$ which is equal to its reversal is determined by its first $n/2$ letters, hence there are at most $n3^{n/2}$ words whose reversal may be equal to one of their shifts. The result follows.
\end{remark}

\section*{Acknowledgements}

We thank Andrew Howroyd for his interest in the sequence $(\# \calB_n)$ and his proof of Corollary~\ref{cor:howroyd}. We also thank anonymous referees for their careful reading of a submitted version of this paper and their suggestions for improvements. SR acknowledges the support of the Fondation Sciences Math\'ematiques de Paris. PT acknowledges partial support from Agence Nationale de la Recherche, Grant Number ANR-14-CE25-0014 (ANR GRAAL).

\Addresses
\end{document}